\documentclass{article}
\usepackage[a4paper, scale = 0.7, centering]{geometry}
\usepackage{amssymb,amsmath,amsthm}
\usepackage{mathtools}
\usepackage{graphicx,epstopdf}
\usepackage[usenames]{color}
\usepackage{url}
\usepackage{algorithm,algorithmic}
\usepackage{verbatim}
\usepackage[utf8]{inputenc}

\usepackage{graphicx}
\graphicspath{{./pics/}}
\usepackage{subfig}
\usepackage{flafter}

\usepackage{float}

\graphicspath{{./pics/}}

\newtheorem{theorem}{Theorem}[section]
\newtheorem{lemma}{Lemma}[section]

\newtheorem{remark}{Remark}[section]

\numberwithin{equation}{section}

\title{Acousto-Electric Tomography with Total Variation Regularization}

\author{Bolaji James Adesokan\thanks{DTU Compute, Technical University of Denmark, 2800 Kgs. Lyngby, Denmark (bcsj@dtu.dk, kiknu@dtu.dk)}\and Bj\o rn Jensen\footnotemark[1] \thanks{Corresponding author}\and Bangti Jin\thanks{Department of Computer Science, University College London, Gower Street, London WC1E 6BT, UK (b.jin@ucl.ac.uk, bangti.jin@gmail.com)}\and Kim Knudsen\footnotemark[1]}

\begin{document}

\maketitle

\begin{abstract}
We study the numerical reconstruction problem in acousto-electric tomography of recovering
the conductivity distribution in a bounded domain from interior power density data. 
We propose a numerical method for recovering discontinuous
conductivity distributions, by reformulating it as an optimization problem with $L^1$ fitting and
total variation penalty subject to PDE constraints. We establish continuity and differentiability
results for the forward map, the well-posedness of the optimization
problem, and present an easy-to-implement and robust numerical method based on successive linearization, smoothing and iterative reweighing.
Extensive numerical experiments are presented to illustrate the feasibility of the proposed approach.
\end{abstract}

\section{Introduction}

Acousto-electric tomography (AET) is one promising hybrid data imaging modality that has received
increasing interest in the last decade \cite{ZhangWang:2004,Ammari:2008,WidlakScherzer:2012}. It exploits the acousto-electric effect \cite{Korber1909,FoxHerzfeldRock1946,JossinetLavandierCathignol1998}, i.e., the occurrence of small, localized changes in conductivities in the interior of a body due to a focused ultrasonic wave generated in the exterior. When the ultrasound is induced in combination with Electrical Impedance Tomography (EIT), one may  reconstruct the interior power density data from EIT measurements \cite{Ammari2007},
which can then be used for imaging the internal conductivity.
The availability of internal data greatly improves the resolution and contrast in the reconstructions when compared with conventional EIT. 
 
Mathematically, AET  can be formulated as follows. Let $\Omega \subset \mathbb{R}^d $, $ d = 2,3$, be an open bounded domain with a Lipschitz boundary $ \Gamma $. The conductivity inside $\Omega$ is denoted by $\sigma.$ 
Applying a current field $f_j$, $j=1,\ldots,n,$ to $\Gamma$ induces an interior electric potential $u_j$ given as the solution to
\begin{equation*}
        -\nabla \cdot(\sigma\nabla u_j) = 0  \text{ in $ \Omega $}, \quad\mbox{with }
        \sigma \frac{\partial u_j}{\partial \nu} = f_j  \text{ on $ \Gamma $.}
\end{equation*}
Here $ \nu $ denotes the unit outward normal vector on $ \Gamma $. The AET forward problem  is now  to find the interior power density defined by $ H_j(\sigma) = \sigma|\nabla u_j|^2 $ from knowledge of $\sigma$, $j=1,\ldots,n$. The corresponding inverse problem is, from knowledge of $H_j(\sigma)$ (and $f_j),$ to recover the conductivity $\sigma.$
The issues of unique recovery and stability have been extensively studied (see, e.g., \cite{Bal2013,AlbertiCapdeboscqBook} and references therein). 
For example, in two spatial dimensions uniqueness is known for three properly chosen boundary conditions ensuring a non-vanishing Jacobian condition in the interior 
\cite{AlbertiCapdeboscqBook}.

Various aspects of numerical reconstruction in AET have been considered \cite{Ammari2007,
GebauerScherzer:2008,KuchmentKunyansky:2011,HoffmannKnudsen2014,BalKnudsen:2018}. Ammari et al.\ \cite{Ammari2007} proposed an
algorithm for recovering the conductivity from multiple power densities, which essentially relies
on a perturbation approach and thus most suitable for small inclusions (relative to a known background).
Capdesboscq et al. \cite{Capdeboscq2009}
proposed two optimal control formulations for reconstructing the conductivity and presented numerical results to illustrate the
effectiveness of the two approaches. Bal et al. \cite{Bal2012} proposed a Levenberg-Marquardt iteration for numerical inversion, analyzed the convergence and regularizing  properties of the algorithm in a Hilbert
space setting (i.e., $H^s(\Omega)$ with $s>d/2$), under the assumptions that the derivative of the forward operator is injective and the data is noise-free. The case of limited angle boundary data was considered in \cite{HubmerKnudsenLiSherina2018}.
The explicit formulation of the reconstruction problems as a regularized  output least-squares problem was done in \cite{AdesokanKnudsenKrishnanRoy2018} and taken further to Perona-Malik type edge enhancing regularizers \cite{Roy-Borzi}. 
See also, e.g., \cite{BalGuoMonard:2014,MonardRim:2018}, for the case of anisotropic conductivities.

Total variation penalty is extremely popular for image processing. Since the seminal work \cite{RudinOsherFatemi:1992}
on image denoising, it has also been widely applied to solving inverse problems. There are several
works on nonlinear parameter identifications for PDEs with total variation penalty, e.g., underground water flow
\cite{chen_augmented_1999}, electrical impedance tomography \cite{HinzeKaltenbacher:2018} and
quantitative photo-acoustic tomography \cite{HannukainenHyvonenMajanderTarvainen:2016}, which have inspired the present work.
Note that in these works, typically an $L^2(\Omega)$ fitting term
is employed, which is well suited for theoretical considerations and numerical computations.

The main focus of this work is to reconstruct conductivities that are mostly piecewise
constants using a PDE constrained optimal control formulation with a total variation penalty. Our main contributions are as follows. First, we provide a proper functional analytic setting for
the reconstruction problem with discontinuous conductivity distributions. This is achieved by carefully analyzing
the parameter-to-data map, e.g., continuity and differentiability in Theorem \ref{prop:deriv-H}. The analysis relies crucially on
the $W^{1,q}(\Omega)$ regularity of the state variable $u(\sigma)$ in Theorem \ref{thm:reg}. Second, we formulate the reconstruction problem as an optimization problem on an $L^1(\Omega)$ fitting term and total variation penalty term: 
\begin{equation*}
\min_{\sigma} \mathcal J_\beta(\sigma) = \sum_{j=1}^n\|H_j(\sigma)-z_j\|_{L^1(\Omega)} + \beta |\sigma|_{\textup{TV}}, 
\end{equation*}
over a suitable admissible set, and analyze the well-posedness of the formulation,
e.g., existence and stability of minimizers in Theorems \ref{thm:minimizer} and \ref{thm:min-stab}. Third, we
describe an easy-to-implement numerical algorithm based on recursive linearization, smoothing and iterative reweighing, cf.
Algorithm \ref{alg:euclid}. Fourth, we present extensive
numerical experiments with full and partial data to illustrate the effectiveness of the proposed approach. Further, we analyze
the convergence of the discrete approximations by the Galerkin finite element method for both nonlinear and linearized models using suitable
$W^{1,q}(\Omega)$ estimates on the finite element approximations in Lemmas \ref{lem:dis-Meyers} and \ref{lem:fem-W1q}.

The paper is organized as follows. In Section \ref{sec:prelim}, we recall
preliminary results on function spaces. Then in Section
\ref{sec:forward}, we discuss mapping properties of the solution operator and
parameter-to-data map, e.g., continuity and differentiability. In Section \ref{sec:rec},
we formulate the AET reconstruction into a PDE constrained optimization problem and
analyze its analytic properties. In Section \ref{sec:alg}, we describe an algorithm for the numerical
solution of the optimization problem.
Last, in Section \ref{sec:numer}, we present extensive numerical results to illustrate
the effectiveness of the reconstruction technique. In Appendix \ref{sec:fem}, we give
a finite element convergence analysis. Throughout, the notation $C$ denotes a generic constant
which may differ at each occurrence, but it is always independent of the mesh size $h$
and other quantities under study.

\section{Preliminaries on function spaces}\label{sec:prelim}

This part reviews basic functional analytic tools and also fix the notation.

\subsection{Sobolev spaces}

First, we recall Sobolev spaces, which will be used extensively below.
For any multi-index $\alpha\in \mathbb{N}^d$, $|\alpha|$ denotes the sum of all components.
Given a domain $\Omega\subset\mathbb{R}^d$ with a Lipschitz continuous boundary $\Gamma$,
for any $m\in \mathbb{N}$, $1\leq p\leq \infty$, we follow \cite{Adams2003} and define the
Sobolev space $W^{m,p}(\Omega)$ by
\begin{equation*}
W^{m,p}(\Omega)=\{u\in L^{p}(\Omega): D^{\alpha}u\in L^{p}(\Omega) \text{ for } 0\leq	|\alpha|\leq m \}.
\end{equation*}
It is equipped with the norm
\begin{equation*}
  \|u\|_{W^{m,p}(\Omega)} = \begin{cases}
    \left(\sum\limits_{0\leq |\alpha|\leq m}\|D^{\alpha}u\|_{L^p(\Omega)}^{p}\right)^{\frac{1}{p}}, & \text{ if }1\leq p<\infty,\\
    \max\limits_{0\leq |\alpha|\leq m}\|D^{\alpha}u\|_{L^\infty(\Omega)} ,& \text{ if } p=\infty.
  \end{cases}
\end{equation*}
The space $W_{0}^{m,p}(\Omega)$ is the closure of $C_c^{\infty}(\Omega)$ in $W^{m,p}(\Omega)$. Its dual space is denoted by $W^{-m,p'}(\Omega)$, with ${1}/{p}+{1}/{p'}=1$, i.e., $p'$ is the conjugate exponent of $p$. Also we use $H^{m}(\Omega)=W^{m,2}(\Omega)$, and $H_0^m(\Omega)=W_0^{m,2}(\Omega)$. We denote by $(W^{m,p}(\Omega))'$ the dual space of $W^{m,p}(\Omega)$.

The bracket $(\cdot,\cdot)$ denotes the $L^2(\Omega)$ inner product,  $({\cdot,\cdot})_{L^2(\Gamma)}$ the $L^2(\Gamma)$ inner product, and $ \langle{\cdot,\cdot}\rangle $ dual pairing. The space $V\equiv H_\diamond^1(\Omega) \subset H^1(\Omega) $ consists of functions with zero mean on the boundary, that is
$$H_\diamond^1(\Omega) := \{ u\in H^1(\Omega): \int_\Gamma u\,ds = 0 \} \simeq H^1(\Omega)/\mathbb{R}.$$
We denote by $H^{-\frac12}(\Gamma)$ the dual space of $H^{\frac12}(\Gamma)$, and $ H_\diamond^{-\frac12}
(\Gamma):=\{v\in (H^{\frac12}(\Gamma))': \langle v,1\rangle=0\}$.
Functions in $ H_\diamond^1(\Omega) $ satisfy the Poincar\'{e} type inequality
\begin{equation}\label{eqn:poincare}
 \|v\|_{H^1(\Omega)} \leq C\|\nabla v\|_{L^2(\Omega)}, \quad \forall v \in H_\diamond^1(\Omega),
\end{equation}
which can be derived by standard compactness arguments (see, e.g., \cite[Section 5.4]{Attouch2006}).

\subsection{Space of bounded variation}
Below we shall use the total variation penalty in the regularized reconstruction, for which the proper function
space is the space of bounded variation $BV(\Omega)$. We only describe some basic properties, and refer
interested readers to \cite{AmbrosioFuscoPallara:2000,Attouch2006,EvansGariepy:2015} for details.
The space $ BV(\Omega)$ consists of functions $v\in L^1(\Omega)$ whose distributional derivative $Dv$ is a Radon measure, i.e.,
\begin{equation*}
  BV(\Omega) = \{v\in L^1(\Omega): \,|v|_{\rm TV} <\infty \},
\end{equation*}
where the total variation $|v|_{\rm TV}$ is defined by
\begin{equation*}
  |v|_{\rm TV} = \int_\Omega d|Dv| =
  \sup\left\{\int_\Omega v\operatorname{div}\psi\,dx: \psi \in C_c^1(\Omega)^d, |\psi(x)|\leq 1\right\},
\end{equation*}
and $|\cdot|$ denotes the Euclidean norm of vectors in $\mathbb{R}^d$.
The space $ BV(\Omega)$  is a Banach space when equipped with the norm
\begin{equation*}
  \|v\|_{BV(\Omega)} = \|v\|_{L^1(\Omega)} + |v|_{\rm TV}.
\end{equation*}
There are several different notions of convergence on the space $BV(\Omega)$.
Besides the strong and the weak-$ \ast $ topology on $ BV(\Omega) $, there is the \emph{intermediate} topology (also known as strict convergence), which is in between the two and characterized by strong $ L^1(\Omega) $ convergence and convergence of $ |\cdot|_{\rm TV} $ in $ \mathbb{R} $, i.e. a sequence $ \{v_n\} \subset BV(\Omega) $ converges in the intermediate sense to $ v \in BV(\Omega) $ if $ v_n \to v $ in $ L^1(\Omega) $ and $ |v_n|_{\rm TV} \to |v|_{\rm TV} $ as $ n \to \infty $ \cite[Definition 10.1.3, p. 374]{Attouch2006} \cite[Definition 3.14, p. 125]{AmbrosioFuscoPallara:2000}. The intermediate convergence is very useful in practice, and we will equip $BV(\Omega)$ with this topology in the sequel.

The following results on $BV(\Omega)$ are very useful. The first assertion can be found at \cite[Theorem 10.1.4, p. 378]{Attouch2006}, the second at \cite[Theorem 5.2, p. 199]{EvansGariepy:2015}, and the third at \cite[Theorem 10.1.2, p. 375]{Attouch2006}. For Assertion (i), the domain $\Omega$ has to satisfy suitable regularity, e.g., extension domain in the sense of \cite[Definition 3.20, p. 130]{AmbrosioFuscoPallara:2000} (see, e.g., \cite[Theorem 3.23, p. 132]{AmbrosioFuscoPallara:2000}). Any open set with a compact Lipschitz boundary is an extension domain \cite[Proposition 3.21, p. 131]{AmbrosioFuscoPallara:2000}.
\begin{lemma} \label{lem:BV-props}
The following properties hold on the space $BV(\Omega)$.
\begin{itemize}
    \item[$\rm(i)$]  The space $ BV(\Omega) $ embeds compactly into $ L^p(\Omega) $ for $ p < \frac{d}{d-1} $.
    \item[$\rm(ii)$] The total variation is lower semi-continuous with respect to the
convergence in $L^1(\Omega)$, i.e., if $\{v_n\}\subset BV(\Omega)$
and $v_n\to v$ in $L^1(\Omega)$, we have
\begin{equation*}
  |v|_{\rm TV} \leq\liminf_{n\to\infty} |v_n|_{\rm TV}.
\end{equation*}
\item[$\rm(iii)$] The space $C^\infty(\overline{\Omega})$ is dense in $BV(\Omega)$ with respect to the convergence in the intermediate sense.
\end{itemize}
\end{lemma}

\section{Properties of the AET forward map}\label{sec:forward}

In this section we establish continuity and differentiability of the forward map in AET. For sufficiently regular conductivities, such results are well-known in the literature, but the issues become more delicate when non-smooth conductivities are considered. 

For a fixed $ \lambda \in (0,1) $ we define the set
\begin{equation} \label{asmp:1}
    \mathcal{S}=\{\sigma: \lambda\leq \sigma\leq \lambda^{-1}\ \text{a.e. in } \Omega\}
    \tag{A.1}
\end{equation}
and we assume that $\sigma \in \mathcal S.$ This assumption is physically reasonable and in particular it implies that $ \sigma \in L^\infty(\Omega) $.
We write $(\mathcal{S},\tau)$ to denote the set endowed with a topology $\tau$.

Given $f\in H_\diamond^{-\frac12}(\Gamma)$ we consider the PDE problem
\begin{equation}\label{eqn:forward}
    \left\{\begin{aligned}
      -\nabla\cdot(\sigma\nabla u) & = 0, && \text{in $\Omega$},\\
      \sigma\frac{\partial u}{\partial\nu} & = f, && \text{in $ \Gamma $}.
    \end{aligned}\right.
\end{equation}
The weak formulation of problem \eqref{eqn:forward} is to find $u\in V$ such that
\begin{equation*}
    (\sigma\nabla u,\nabla \psi) = \langle f,\psi \rangle \quad \forall \psi\in V.
\end{equation*}
Clearly, $ f $ must satisfy the compatibility condition $ \langle f, 1 \rangle = 0 $, which we shall assume throughout the rest of the paper independent of the space in which $ f $ is taken. By the Poincar\'{e} inequality \eqref{eqn:poincare}, Lax-Milgram theorem implies that \eqref{eqn:forward} has a unique weak solution $ u = u(\sigma) \in V$ and
\begin{equation*}
    \|u(\sigma)\|_{H^1(\Omega)} \leq C\|f\|_{H^{-\frac12}(\Gamma)},
\end{equation*}
where the constant $C$ depends on $\lambda$ and the domain $\Omega$ but is independent of $\sigma$. For a fixed $f$, we shall suppress the dependence of the solution $u$ on $f$, and only indicate its dependence on $\sigma$ by writing $u(\sigma)$.

Next we recall a $W^{1,q}(\Omega)$ regularity result for elliptic problems; see \cite{gallouet1999} for a proof. This result will play an important role in the analysis below.
\begin{theorem}\label{thm:reg}
Let $\sigma \in \mathcal S$ and suppose $g\in L^q(\Omega)$, $h\in L^q(\Omega)^d$ and $f\in (W^{1-\frac1q,q}(\Gamma))'$ with $\int_\Gamma f\,ds + \int_\Omega g\,dx=0.$ Then there exists a constant $Q>2,$ such that for any $q\in(2,Q)$,   the problem 
\begin{equation*}
    \left\{\begin{aligned}
      -\nabla\cdot(\sigma\nabla u)&= g + \nabla \cdot h, && \text{in $ \Omega $,}\\
      \sigma\frac{\partial u}{\partial\nu} & = f, && \text{on $\Gamma$,}
    \end{aligned}\right.
\end{equation*}
has a unique weak solution $u(\sigma) \in W^{1,q}(\Omega)$ satisfying
\begin{equation*}
  \|u\|_{W^{1,q}(\Omega)} \leq C(\|f\|_{(W^{1-\frac1q,q}(\Gamma))'}+\|g\|_{L^q(\Omega)}+\|h\|_{L^q(\Omega)^d}).
\end{equation*}
The constant $Q = Q(\lambda,d)$ depends only on the domain $\Omega$, the spatial dimension $d$ and the constant $\lambda$ and the constant $C$ depends only on $\Omega$, $\lambda$, $d$ and $q$.
\end{theorem}

\begin{remark}
The parameter $Q$ depends on the regularity of the domain $\Omega$. If the domain $\Omega$ is of class $C^1$, then $Q(\lambda,d)\rightarrow\infty$ as $\lambda\rightarrow1$ \cite{Groger:1989}. For a general Lipschitz domain, e.g., polyhedrons, there also always exists some $Q(\lambda,d)>2$
for any $\lambda<1$ \cite[Section 5]{Groger:1989}. In the two-dimensional case, the optimal exponent $Q$ was discussed in \cite{AstalaEaracoSzekelyhidi:2008,NesiPalombaro:2014}.
\end{remark}


\subsection{The continuity and differentiability of the solution map}

In order to analyze the forward map $\sigma \mapsto H(\sigma),$ we first address the continuity of the solution operator $\sigma\mapsto u(\sigma)$. This map is identical with that for EIT and has been extensively studied in various function spaces \cite{chen_augmented_1999,JinMaass:2012,DunlopStuart:2016}. Hence, we only sketch the proof for the convenience of readers.

\begin{lemma} \label{lem:cont-of-sol-op}
Let $ \{\sigma_k\}\subset \mathcal{S} $ satisfy $ \sigma_k \to \sigma^\ast $ in $ L^1(\Omega)$. Then the following statements hold.
\begin{itemize}
    \item[$\rm (i)$] If $f\in H_\diamond^{-\frac12}(\Gamma)$, then there exists a subsequence of $\{u(\sigma_k)\}$ convergent to $u(\sigma^*)$ in $H^1(\Omega)$.
    \item[$\rm(ii)$] If $f\in (W^{1-\frac{1}{r},r}(\Gamma))'$ for some $r>2$, then $u(\sigma_k)\to u(\sigma^*)$ in $W^{1,q}(\Omega)$ for any $q\in(2,\min(Q,r))$.
\end{itemize}
\end{lemma}
\begin{proof}
For notational simplicity, we denote $u_k\equiv u(\sigma_k)$ and $u^*\equiv u(\sigma^*)$. Let $w_k=u_k-u^*$. It follows from the weak formulation for $u_k$ and $u^*$ that
\begin{equation}\label{eqn:err-eq}
    ({\sigma_k\nabla w_k,\nabla\phi}) = ({(\sigma^*-\sigma_k)\nabla u^*,\nabla \phi})\quad\forall\phi\in V.
\end{equation}
Now we discuss the two cases separately.

In case (i), letting $\phi=w_k\in V$ gives
\begin{equation*}
    \begin{aligned}
      ({\sigma_k\nabla w_k,\nabla w_k}) & = ({(\sigma^*-\sigma_k)\nabla u^*,\nabla w_k})\\
      & \leq \|\nabla w_k\|_{L^2(\Omega)}\|(\sigma_k-\sigma^*)\nabla u^*\|_{L^2(\Omega)}.
    \end{aligned}
\end{equation*}
By the standard measure theory,  convergence in $L^p(\Omega)$, $p\geq 1$, implies almost everywhere convergence up to a subsequence \cite[Theorem 1.21, p. 29]{EvansGariepy:2015}. Thus, one can extract a subsequence of $\{\sigma_k\}$, still denoted by $\{\sigma_k\}$, that converges almost everywhere in $\Omega$.
This, the trivial inequality $|\sigma_k-\sigma^*|\leq \lambda^{-1}$ a.e. and Lesbesgue's dominated convergence theorem \cite{EvansGariepy:2015} imply
\begin{equation*}
    \lim_{k\to\infty} \|(\sigma_k-\sigma^*)\nabla u^*\|_{L^2(\Omega)}=0.
\end{equation*}
This and the condition $\sigma_k\geq \lambda$ give the desired assertion.

In case (ii), by Theorem \ref{thm:reg}, for any $q\in (2,\min(Q,r))$, there holds
\begin{equation*}
    \|\nabla w\|_{L^{q}(\Omega)}\leq C\|(\sigma_k-\sigma^*)\nabla u^*\|_{L^q(\Omega)}.
\end{equation*}
Since $q<\min(Q,r)$, we can choose $q'\in(q,\min(Q,r))$
and by H\"{o}lder inequality, we obtain
\begin{equation*}
    \|(\sigma_k-\sigma^*)\nabla u^*\|_{L^q(\Omega)} \leq \|\sigma_k-\sigma^*\|_{L^{p'}(\Omega)}\|\nabla u^*\|_{L^{q'}(\Omega)},
\end{equation*}
where the exponent $p'$ satisfies $q'^{-1}+p'^{-1}=q^{-1}$. Since $\sigma_k,\sigma^*\in\mathcal{S}$, we have $|\sigma_k-\sigma^*|\leq \lambda^{-1}$ and thus
\begin{equation*}
    \|\sigma_k-\sigma^*\|_{L^{p'}(\Omega)} \leq \lambda^{\frac{1}{p'}-1}\|\sigma_k-\sigma^*\|_{L^1(\Omega)}^{\frac1{p'}},
\end{equation*}
from which it follows directly
\begin{equation*}
    \lim_{k\to\infty} \|\sigma_k-\sigma^*\|_{L^{p'}(\Omega)}=0.
\end{equation*}
This and the uniform bound on $\|\nabla u^*\|_{L^{q'}(\Omega)}$ from Theorem \ref{thm:reg} imply
\begin{equation*}
  \begin{aligned}
    0&\leq \lim_{k\to\infty}\|\nabla w_k\|_{L^q(\Omega)} \leq \lim_{k\to\infty}\lambda^{-1}\|\sigma_k \nabla w_k\|_{L^q(\Omega)}\\
    & \leq \lim_{k\to\infty} \lambda^{-1}\|(\sigma_k-\sigma^*)\nabla u^*\|_{L^q(\Omega)}=0.
  \end{aligned}
\end{equation*}
This completes the proof of the lemma.
\end{proof}

Next we turn to the differentiability of the solution operator $\sigma\mapsto u(\sigma)$ for a fixed $f$. These properties are important for deriving necessary optimality conditions and developing numerical algorithms. However, one needs to be cautious: although the set $\mathcal{S}$ has an interior point with respect to the $L^\infty(\Omega)$ norm, it does not have any interior point with respect to the $L^p(\Omega)$ norm for any $1\leq p<\infty$. In the latter case, the results below have to be understood with respect to the relative topology. The next result gives the formula for the directional derivative $u'(\sigma)[\kappa]$ of $u(\sigma)$ at $\sigma$ in the direction $\kappa\in L^\infty(\Omega)$.

\begin{lemma}\label{lem:direc-derivative}
For $ \sigma \in\mathcal{S}$, the directional derivative $u'(\sigma)[\kappa]$ solves
\begin{equation} \label{eq:pde-du}
({\sigma\nabla u'(\sigma)[\kappa],\nabla \phi}) = -({\kappa\nabla u(\sigma),\nabla \phi}) \quad \forall \phi\in V.
\end{equation}
The map $u'(\sigma)[\cdot]: L^\infty(\Omega)\to V$ is continuous. If $f\in(W^{1-\frac{1}{r},r}(\Gamma))'$ for some $r>2$, then for any $q\in (2,\min(Q,r))$ and $p>\frac{q\min(Q,r)}{\min(Q,r)-q}$, $u'(\sigma)[\cdot]: L^p(\Omega)\to W^{1,q}(\Omega)$ is continuous.

Further, the following statements hold.
\begin{itemize}
\setlength\itemsep{-.25em}
\item[$\rm (i)$] If $f\in H_\diamond^{-\frac{1}{2}}(\Gamma)$, then $u: (\mathcal{S}, L^\infty(\Omega)) \to H^1(\Omega) $ is Fr\'echet differentiable.
\item[$\rm(ii)$] If $f\in (W^{1-\frac1r,r}(\Gamma))'$ for some $r>2$, then $u: (\mathcal{S},L^p(\Omega)) \to W^{1,q}(\Omega) $ is Fr\'echet differentiable for any $q\in (2,\min(Q,r))$ and $p>\frac{q\min(Q,r)}{\min(Q,r)-q}$.
\end{itemize}
\end{lemma}
\begin{proof}

The expression of the directional derivative follows from straightforward computation. Specifically,  let $ h(t) = \frac{1}{t}(u(\sigma + t\kappa) - u(\sigma)) $, $ t > 0 $. Assume that $ t $ is sufficiently small $ t $ so that $ \lambda/2 < \lambda - t\|\kappa\|_{L^\infty} $. Now by the weak formulation for $u(\sigma+t\kappa)$ and $u(\sigma)$, i.e.,
\begin{equation*}
    \begin{aligned}
        ({(\sigma+t\kappa)\nabla u(\sigma+t\kappa),\nabla\phi}) & = ({f,\phi})_{L^2(\Gamma)}\quad \forall \phi\in V,\\
        ({\sigma\nabla u(\sigma),\nabla\phi}) & = ({f,\phi})_{L^2(\Gamma)}\quad \forall \phi\in V.
    \end{aligned}
\end{equation*}
Taking the difference between the two equations and appealing to the definition of $h(t)$ yield
\begin{equation}\label{eqn:diff}
   (\sigma\nabla h(t),\nabla\phi) = - (\kappa\nabla u(\sigma+t\kappa),\nabla\phi)\quad \forall\phi\in V.
\end{equation}
Letting $\phi=h(t)$ shows that $h(t)$ is uniformly bounded in $H^1(\Omega)$ as $t\to 0^+$. Further, we have $\|(\sigma+t\kappa)-\sigma\|_{L^p(\Omega)}=\lim_{t\to0^+}t\|\kappa\|_{L^p(\Omega)}=0$, for any $p>0$. By Lemma \ref{lem:cont-of-sol-op}, we deduce
\begin{equation*}
    \lim_{t\to0^+}\|u(\sigma+t\kappa)-u(\sigma)\|_{H^1(\Omega)} = 0.
\end{equation*}
Thus letting $t\to0^+$ in \eqref{eqn:diff} yields directly the weak formulation for $u'(\sigma)[\kappa]$. The bound and continuity of the map $u'(\sigma)[\kappa]:L^\infty(\Omega)\to V$ follows similarly.

Now if $f\in (W^{1-\frac{1}{r},r}(\Gamma))'$, by Theorem \ref{thm:reg}, $u(\sigma)\in W^{1,s}(\Omega)$ for any $s\in [2,\min(Q,r))$. Thus, for any $q\in [2,\min(Q,r))$, we can choose any $q'\in(q,\min(Q,r))$, and $p^{\prime-1}+q^{\prime-1}=q^{-1}$, by H\"{o}lder's inequality, there holds
\begin{equation*}
    \|\kappa\nabla u(\sigma)\|_{L^q(\Omega)} \leq \|\kappa\|_{L^{p'}(\Omega)}\|\nabla u(\sigma)\|_{L^{q'}(\Omega)}.
\end{equation*}
Then by Theorem \ref{thm:reg}, the solution $u'(\sigma)[\kappa]$ to \eqref{eqn:diff} belongs to $W^{1,q}(\Omega)$ and
\begin{equation*}
    \|u'(\sigma)[\kappa]\|_{W^{1,q}(\Omega)}\leq c\|\kappa\|_{L^{p'}(\Omega)}\|\nabla u(\sigma)\|_{L^{q'}(\Omega)}.
\end{equation*}
This shows the boundness of the map $u'(\sigma)[\kappa]:L^p(\Omega)\to W^{1,q}(\Omega)$. Since the value of $q'$ can be made arbitrarily close to $\min(Q,r)$, the desired continuity holds for any $p>\frac{q\min(Q,r)}{\min(Q,r)-q}$.

Next we turn to Fr\'{e}chet differentiability. Assertion (i) is known; see, e.g., \cite{kuchment_stabilizing_2012}. Thus, we only prove part (ii).
Let $w(\sigma,\kappa) = u(\sigma + \kappa) - u(\sigma) - u'(\sigma)[\kappa]$. By the preceding argument, for any $q\in[2,\min(Q,r))$, the map $u'(\sigma)[\kappa]: L^p(\Omega)\to W^{1,q}(\Omega)$ is bounded, for any $p>\frac{q\min(Q,r)}{\min(Q,r)-q}$. It suffices to show
\begin{equation*}
\lim_{\|\kappa\|_{L^p(\Omega)}\to 0}\frac{\|w(\sigma,\kappa)\|_{W^{1,q}(\Omega)}}{\|\kappa\|_{L^p(\Omega)}} = 0,
\end{equation*}
Next we take a sufficiently small $ \kappa $ (with $\sigma+\kappa\in\mathcal{S}$). Then the residual $w(\sigma,\kappa)$ satisfies
\begin{equation}\label{eq:pde-rem}
    ({(\sigma+\kappa)\nabla w(\sigma,\kappa),\nabla \phi}) = ({\kappa \nabla u'(\sigma)[{\kappa}],\nabla \phi})\quad \forall \phi\in V.
\end{equation}
By the preceding argument, we have \begin{equation*}
    \|u'(\sigma)[\kappa]\|_{W^{1,q'}(\Omega)}\leq c\|\kappa\|_{L^p(\Omega)},
\end{equation*}
where the exponent $q'\in (q, \min(Q,r))$ is sufficiently close to $q$. Therefore, by choosing $p'$ large such that $p^{\prime-1}+q^{\prime-1}=q$ and H\"{o}lder's inequality, there holds
\begin{equation*}
  \begin{aligned}
    \|\kappa\nabla u'(\sigma)[\kappa]\|_{L^q(\Omega)} & \leq \|\kappa\|_{L^{p'}(\Omega)}\|u'(\sigma)[\kappa]\|_{L^{q'}(\Omega)}\\
    &\leq c\|\kappa\|_{L^{p'}(\Omega)}\|\kappa\|_{L^p(\Omega)}.
 \end{aligned}
\end{equation*}
This and Theorem \ref{thm:reg} imply
\begin{equation*}
 \|w(\sigma,\kappa)\|_{W^{1,q}(\Omega)}\leq c\|\kappa\nabla u'(\sigma)[\kappa]\|_{L^q(\Omega)}
\leq c\|\kappa\|_{L^{p'}(\Omega)}\|\kappa\|_{L^p(\Omega)}.
\end{equation*}
Now the desired assertion follows, in view of the inequality $\|\kappa\|_{L^{p'}(\Gamma)}\leq \lambda^{\frac{p}{p'}-1}\|\kappa\|_{L^p(\Omega)}^{\frac{p}{p'}}$.
\end{proof}

\begin{remark}
If $f\in H_\diamond^{-\frac12}(\Gamma)  $, the map $u(\sigma): L^\infty(\Omega)\to H^1(\Omega)$ is holomorphic \cite[p. 23]{CohenDevore:2015}.
Lemma \ref{lem:direc-derivative} gives a slightly stronger result on the first derivative under the condition  $f\in (W^{1-\frac1r,r}(\Omega))'$.
Our discussions have focused on the $L^p(\Omega)$ spaces, which is suitable for low-regularity penalties, e.g., total variation and $H^{1}(\Omega)$ penalty. More smoothing penalties, e.g., $H^2(\Omega)$, are also adopted  \cite{Bal2012}. Generally, if $ \Omega $ has a $ C^{k+1} $-boundary $\Gamma$, $\sigma,\kappa \in C^k\left(\overline{\Omega}\right) $ and $ f \in H^{k-\frac{1}{2}}(\Gamma)$ with $ k > \frac{d}{2}$, then $ u : H^k(\Omega) \to H^{k+1}(\Omega) $ is also Fr\'echet differentiable. This can be proved similarly \cite{Bal2012}.
\end{remark}

\subsection{The continuity and differentiability of the forward map}

The next result gives the continuity of the forward map $H(\sigma)$. It follows directly from Lemma \ref{lem:cont-of-sol-op}.

\begin{lemma}\label{lem:H-cont}
Let $\{\sigma_k\}\subset\mathcal{S}$ be convergent to $\sigma^*\in\mathcal{S}$ in $L^1(\Omega)$. Then the following statements hold.
\begin{itemize}
    \item[$\rm (i)$] If $f\in  H_\diamond^{-\frac12}(\Gamma)$, then there exists a subsequence of $\{H(\sigma_k)\}$ converging to  $H(\sigma^*)$ in $L^1(\Omega)$.
    \item[$\rm(ii)$] If $f\in (W^{1-\frac1r,r}(\Gamma))'$ for some $r>2$, then $H(\sigma_k)\to H(\sigma^*)$ in $L^{\frac{q}2}(\Omega)$ for any $q\in[2,\min(Q,r))$.
\end{itemize}
\end{lemma}
\begin{proof}
We split the difference $H(\sigma_k)-H(\sigma^*)$ into
\begin{equation}\label{eqn:H-split}
    H(\sigma_k)-H(\sigma^*)=
    (\sigma_k-\sigma^*)|\nabla u^*|^2+\sigma_k(|\nabla u_k|^2-|\nabla u^*|^2),
\end{equation}
and bound the two terms independently. 
We discuss case (i) and (ii) separately.

In case (i), by the $L^1(\Omega)$ convergence, we deduce that there exists a subsequence of $\{\sigma_k\}$, still relabelled as $\{\sigma_k\}$, that converges almost everywhere to $\sigma^*$. Since $|\sigma_k-\sigma^*|\leq\lambda^{-1}$, by Lebesgue's dominated convergence theorem and the condition $\sigma_k\to\sigma^*$ a.e., we deduce
\begin{equation*}
    \lim_{k\to\infty}\|(\sigma_k-\sigma^*)|\nabla u^*|^2\|_{L^1(\Omega)}=0.
\end{equation*}
Meanwhile, by the triangle inequality and H\"{o}lder's inequality, there holds
\begin{align*}
    \|\sigma_k(|\nabla u_k|^2-|\nabla u^*|^2)\|_{L^1(\Omega)}& \leq \lambda^{-1}\| (\nabla u_k-\nabla u^*)\cdot (\nabla u_k+\nabla u^*)\|_{L^1(\Omega)}\\
    &\leq \lambda^{-1}(\|\nabla u_k\|_{L^2(\Omega)}+\|\nabla u^*\|_{L^2(\Omega)})\|\nabla(u_k-u^*)\|_{L^2(\Omega)}.
\end{align*}
The right hand side also tends to zero, in view of Lemma \ref{lem:cont-of-sol-op}(i) and the fact that $\|\nabla u_k\|_{L^2(\Omega)}$ and $\|\nabla u^*\|_{L^2(\Omega)}$ are uniformly bounded independent of $k$.

The proof for case (ii) is similar. In the splitting \eqref{eqn:H-split}, for any $q\in[2,\min(Q,r))$, the first term is now bounded using H\"{o}lder's inequality
\begin{equation*}
    \|(\sigma_k-\sigma^*)|\nabla u^*|^2\|_{L^{\frac{q}{2}}(\Omega)} \leq \|\sigma_k-\sigma^*\|_{L^{p'}(\Omega)}\|\nabla u^*\|_{L^{q'}(\Omega)}
\end{equation*}
where the exponent $q'\in (q,\min(Q,r))$ and the exponent $p'=\frac{qq'}{2(q'-q)}<\infty$. Now the desired assertion follows from $\|\sigma_k-\sigma^*\|_{L^{p'}(\Omega)}\leq \lambda^{-1+1/p'}\|\sigma_k-\sigma^*\|_{L^1(\Omega)}^{1/p'}$, cf. the proof of Lemma \ref{lem:cont-of-sol-op}. Similarly, by the triangle inequality and H\"{o}lder's inequality, the second term can be bounded by
\begin{align*}
    \|\sigma_k(|\nabla u_k|^2-|\nabla u^*|^2)\|_{L^{\frac{q}{2}}(\Omega)}& \leq \lambda^{-\frac{q}{2}}\| (\nabla u_k-\nabla u^*)\cdot (\nabla u_k+\nabla u^*)\|_{L^{\frac{q}{2}}(\Omega)}\\
    &\leq \lambda^{-\frac{q}{2}}(\|\nabla u_k\|_{L^q(\Omega)}+\|\nabla u^*\|_{L^q(\Omega)})\|\nabla(u_k-u^*)\|_{L^q(\Omega)}.
\end{align*}
This and Lemma \ref{lem:cont-of-sol-op}(ii) complete the proof of the lemma.
\end{proof}

Last, we give the directional derivative $H'(\sigma)[\kappa]$ of the forward map
$H(\sigma)$ and its adjoint:
\begin{theorem}\label{prop:deriv-H}
The directional derivative $H'(\sigma)[{\kappa}]$ is given by
\begin{equation*}
 H'(\sigma)[{\kappa}] = \kappa|\nabla u(\sigma)|^2 + 2\sigma\nabla u(\sigma)\cdot\nabla u'(\sigma)[\kappa].
\end{equation*}
For any $q\in [2,Q)$ and $\kappa\in L^\infty(\Omega)$, if $f\in (W^{1-\frac{1}{q},q}(\Gamma))'$, then $H'(\sigma)[\kappa] \in L^{\frac q2}(\Omega),$
and the adjoint $H'(\sigma)^*\zeta$ is given by
\begin{equation*}
  H'(\sigma)^*\zeta = |\nabla u|^2\zeta + 2\nabla u\cdot \nabla v,
\end{equation*}
where $v=v(\zeta)$ solves the problem
\begin{equation*}
    ({\sigma\nabla v,\nabla \phi}) = -({\sigma\zeta\nabla u,\nabla \phi}) \quad \forall \phi\in V.
\end{equation*}
\end{theorem}
\begin{proof}
The formula for $H'(\sigma)[\kappa]$ is direct to derive. By Theorem \ref{thm:reg}, $u(\sigma)\in W^{1,q}(\Omega)$. In the source term $-\nabla\cdot(\kappa\nabla u(\sigma))$, $\kappa\nabla u(\sigma)\in L^q(\Omega)$, and applying Theorem \ref{thm:reg} again gives
$u'(\sigma)[\kappa]\in W^{1,q}(\Omega)$. Then by H\"{o}lder's inequality, we obtain the first assertion. 

To get the adjoint $H'(\sigma)^*$, we employ the weak formulations for $u'(\sigma)[\kappa]$ and $v $, i.e.,
\begin{equation*}
\begin{aligned}
  ({\sigma\nabla u'(\sigma)[\kappa],\nabla\phi}) & = -({\kappa\nabla u(\sigma),\nabla\phi}) &&\forall\psi\in V,\\
  ({\sigma\nabla v,\nabla \phi}) & = -({\zeta\sigma\nabla u(\sigma),\nabla \phi}) &&\forall \phi\in V.
\end{aligned}
\end{equation*}
Thus, there holds
$
    ({\kappa\nabla u(\sigma),\nabla v}) = ({\zeta\sigma\nabla u(\sigma),\nabla u'(\sigma)[\kappa]}).
$
Consequently,
\begin{equation*}
    \begin{aligned}
      ({H'(\sigma)^*[\zeta],\kappa}) & = ({\zeta, H'(\sigma)[\kappa]})\\
      & = ({\zeta, \kappa|\nabla u(\sigma)|^2+2\sigma\nabla u(\sigma)\cdot\nabla u'(\sigma)[\kappa]})\\
      & = ({\zeta|\nabla u(\sigma)|^2,\kappa})
        +2({\nabla u(\sigma)\cdot\nabla v,\kappa}).
    \end{aligned}
\end{equation*}
By the definition of the adjoint, we obtain the formula for $H'(\sigma)^*[\zeta]$.
\end{proof}


The next result gives the compactness of the forward map $H(\sigma)$. Hence, the nonlinear inverse problem is indeed ill-posed, and regularization is needed for stable reconstruction.
\begin{theorem}\label{prop:compact}
If $f\in (W^{1-\frac1r,r}(\Gamma))'$ for some $r>2$, then the map $H(\sigma): (\mathcal{S},BV)\to L^1(\Omega)$ is compact.
\end{theorem}
\begin{proof}
By Lemma \ref{lem:BV-props}, the space $BV(\Omega)$ embeds compactly into $L^1(\Omega)$. Then by Lemma \ref{lem:H-cont}, the map $\sigma\mapsto H(\sigma)$ is continuous from $(\mathcal{S},L^1(\Omega))$ to $L^1(\Omega)$, under the given regularity on $f$. Hence,
$H(\sigma): (\mathcal{S},BV(\Omega))\to L^1(\Omega)$ is compact.
\end{proof}

\section{Regularized problem} \label{sec:rec}

Now we discuss the well-posedness of the optimization problem arising in regularized reconstruction
by means of a total variation penalty. Like before, we focus our discussion on one single dataset, i.e., $n=1$,
since the extension to multiple datasets is easy. Let $ z = H(\sigma^\ast) $ be the power density data
corresponding to the true conductivity $ \sigma^\ast $, possibly corrupted by noise. The optimization problem reads:
\begin{align} \label{eq:opt:problem}
        \operatornamewithlimits{min}_{\sigma \in \mathcal{A}}\ &\big\{\mathcal{J}_\beta(\sigma) = \|H(\sigma) - z\|_{L^1(\Omega)} + \beta|\sigma|_{\rm TV}\big\}
    \end{align}
where $ H(\sigma) = \sigma |\nabla u(\sigma)|^2 $ is the forward map analyzed in Section \ref{sec:forward}, and $\beta>0$ is a regularization parameter controlling the tradeoff between the two terms. 
The admissible set $ \mathcal A $ is given by
\begin{equation*}
    \mathcal A := \{ \sigma \in BV(\Omega): \lambda \leq \sigma \leq \lambda^{-1}\ \text{a.e. in $ \Omega $} \} = BV(\Omega)\cap\mathcal S.
\end{equation*}

In the model \eqref{eq:opt:problem}, the $ L^1 $ fidelity is motivated by analytical considerations: for any $\sigma\in \mathcal{A}$, the power density $H(\sigma)$ is only ensured to be $L^1(\Omega)$,
but not the usual $L^2(\Omega)$, cf. Lemma \ref{lem:H-cont}.  Statistically speaking, $ L^1 $ fitting is
robust to outliers in the data, and popular for handling impulsive noise. 

Now we prove that the functional $\mathcal{J}_\beta $ is (sequentially) continuous with respect to
the intermediate convergence and that there exists a minimizer to problem \eqref{eq:opt:problem}.

\begin{lemma}\label{thm:slsc}
Let $ f \in (W^{1-\frac{1}{r},r}(\Gamma))'$ for some $r>2$, then $\mathcal{J}_\beta $ is continuous on $ \mathcal A $ in the intermediate topology of $ BV(\Omega) $.
\end{lemma}
\begin{proof}
Let $ \{\sigma_k\} \subset \mathcal{A} $ be a sequence convergent in the intermediate topology. Then by
definition it converges in $ L^1(\Omega) $ and $ |\sigma_k|_{\rm TV} \to |\sigma^\ast|_{\rm TV}$.
Meanwhile, by Lemma \ref{lem:H-cont}, $ H(\sigma_k) \to H(\sigma^\ast) $. Thus we obtain the desired continuity.
\end{proof}

\begin{theorem} \label{thm:minimizer}
For any $f\in  H_\diamond^{-\frac12}(\Gamma)$, Problem \eqref{eq:opt:problem} has at least one minimizer.
\end{theorem}
\begin{proof}
Since $ \mathcal{J}_\beta $ is bounded from below by zero, there exists a minimizing sequence $ \{\sigma_k\} \subset \mathcal{A} $ such that
$\lim_{k\to \infty} \mathcal{J}_\beta(\sigma_k) = \inf_{\sigma\in\mathcal{A}} \mathcal{J}_\beta(\sigma). $
This implies $ |\sigma_k|_\textup{TV} \leq C $ for some $ C > 0 $. By the $ L^\infty $ bound on the
admissible set $ \mathcal{A} $ and boundedness of $ \Omega $, we have $ \|\sigma_k\|_{BV(\Omega)} \leq C $.
By the uniform bound of $ \sigma_k $ in $ BV(\Omega) $ and the compact embedding of $ BV(\Omega) $ into $ L^1(\Omega) $
in Lemma \ref{lem:BV-props}(i), there exists a subsequence, again relabeled as $ \sigma_k $, convergent to some
$ \sigma^\ast$ in $L^1(\Omega) $. Now by Lemma \ref{lem:H-cont}(i), 
$ H(\sigma_k) \to H(\sigma^\ast) $ in $L^1(\Omega)$ (possibly by first passing to an a.e. pointwise convergent subsequence). The rest of the proof follows as Lemma \ref{thm:slsc}, except that we only obtain $ |\sigma^\ast|_{\rm TV} \leq \liminf_{k\to \infty}|\sigma_k|_{\rm TV}$.
\end{proof}
\begin{remark}
Note that the existence of a minimizer requires only $f\in  H_\diamond^{-\frac{1}{2}}(\Gamma) $, when compared
with the sequential lower semi-continuity, since the proof does not require intermediately convergent
subsequences in $ BV(\Omega) $.
Now we briefly mention the optimality conditions for problem \eqref{eq:opt:problem}. 
By means of the adjoint method, formally we derive the following optimality system: with $\zeta := \frac{\sigma|\nabla u|^2 - z}{|\sigma|\nabla u|^2 - z|}$ and $u$ solving \eqref{eqn:forward}, the governing equation for $\sigma$ reads
\begin{align*}
  \left\{ \begin{aligned}
        -\nabla \cdot\left(|\nabla \sigma|^{-1}\nabla \sigma\right) &= -\beta^{-1}H'(\sigma)^\ast[\zeta] && \text{in $ \Omega $}, \\
        |\nabla \sigma|^{-1}\partial_\nu \sigma &= 0 && \text{on $ \Gamma $}.
    \end{aligned}\right. \label{eq:1-laplace}
\end{align*}
This is a 1-Laplace equation for $\sigma$. It is known that a BV solution to such problems exist only under certain
conditions on the source term and boundary condition \cite{Mercaldo2009}. 
\end{remark}

In view of the lower semicontinuity of the functional, it is straightforward to derive the following stability and consistency results. We refer to \cite{ScherzerGrasmair:2009,SchusterKaltenbacher:2012,ItoJin:2015} for a proof in the general case.

\begin{theorem}\label{thm:min-stab}
The following statements hold.
\begin{itemize}
    \item[$\rm(i)$] Let the sequence $\{z_j\}$ be convergent to $z$ in $L^1(\Omega)$, and $\{\sigma_j\}$ the corresponding minimizer to the functional $\mathcal J_\beta$ with $z_j$ in place of $z$. Then the sequence $\{\sigma_j\}$ contains a subsequence convergent to a minimizer of $\mathcal J_\beta$ in the intermediate topology in $BV(\Omega)$.
    \item[$\rm(ii)$]Let $\{\delta_j\}\subset\mathbb{R}^+$ with $\delta_j\to0^+$, $\{z_j\}$ be a sequence satisfying $\|z_j-z^\ast\|_{L^1(\Omega)}=\delta_j$ for some exact data $z^\ast$, and $\sigma_j$ be a minimizer to the functional $\mathcal J_{\beta_j}$ with $z_j$ in place of $z^\ast$. If $\beta_j$s satisfy    \begin{equation*}
        \lim_{j\to\infty}\beta_j = 0 \quad\mbox{and}\quad \lim_{j\to\infty}\frac{\delta_j}{\beta_j } =0,
    \end{equation*}
    then the sequence $\{\sigma_j\}$ contains a subsequence convergent to an $|\cdot|_{\rm TV}$-minimizing solution $\sigma^\dag$ in the intermediate topology $BV(\Omega)$.
\end{itemize}
\end{theorem}

\begin{remark}
The functional $\mathcal{J}_\beta$ is generally not convex, similar to the quadratic fitting in \cite{Capdeboscq2009}.
Specifically, let $ j(\sigma) = \int_\Omega |\sigma|\nabla u(\sigma)|^2 - z|\,dx$. Fix $\sigma\in\mathcal{A}$, $ z $
and $ f $ such that $ \sigma|\nabla u(\sigma)|^2 < z $ a.e. in $ \Omega $. Since $ u(\alpha \sigma) = \alpha^{-1}u(\sigma) $, we have
$j(\alpha\sigma) = \int_\Omega (z-\alpha^{-1}\sigma|\nabla u(\sigma)|^2)\,d x.$
Since $ \alpha \mapsto \alpha^{-1} $ is convex on $ \mathbb{R}_+ $, the map $ \alpha \mapsto j(\alpha\sigma) $ is concave, and thus, $ j $ is not convex in $ \sigma $. 
\end{remark}

\section{Numerical algorithm}\label{sec:alg}
In this part, we describe an algorithm for Problem \eqref{eq:opt:problem}. Problem \eqref{eq:opt:problem} is numerically challenging to solve since it involves a nonlinear forward map $H(\sigma)$, and the functional $\mathcal{J}_\beta$ is nonsmooth and nonconvex. We develop
an easy-to-implement and robust algorithm using the following approximations at each outer iteration:
recursively linearizing the operator $ H(\sigma)$, smoothing the nonsmooth term and approximating the $ L^1 $-norms by quadratic terms. The resulting intermediate constrained quadratic optimization problems are then solved by the conjugate gradient method.

\subsection{Derivation}
First, we tackle the nonconvexity with the fitting term $\|H(\sigma)-z\|_{L^1(\Omega)}$ by linearizing the
forward map $H(\sigma)$. Consider a small change $ \kappa $ from a fixed  $ \sigma $ (e.g., current
approximation). Then by Theorem \ref{prop:deriv-H}, $H(\sigma+\kappa) \approx H(\sigma) + H'(\sigma)\kappa $.
Hence, for a fixed $ \sigma $, we obtain a linearized functional $J_{\sigma,\beta}(\kappa) $ defined by
\begin{equation} \label{eq:linearized-functional}
    \mathcal{J}_\beta(\sigma+\kappa) \approx J_{\sigma,\beta}(\kappa) = \|H'(\sigma)\kappa - d_\sigma\|_{L^1(\Omega)} + \beta|\sigma+\kappa|_{\rm TV},
\end{equation}
with $ d_\sigma = z-H(\sigma) $.  Accordingly, we define a new admissible set $\mathcal A_\sigma$ by
$\mathcal{A}_\sigma := \{\kappa \in BV(\Omega): \sigma + \kappa \in \mathcal A \}$ to accommodate the box constraint on $\sigma$. This step gives rise to a nonsmooth but convex functional $J_{\sigma,\beta}(\kappa)$ (in the increment $\kappa$), since both
TV seminorm and $ L^1(\Omega)$ fitting are nondifferentiable. The nonsmoothness renders the numerical
treatment inconvenient, and there are several possible strategies to handle this. Below, we employ the commonly used smoothing:
\begin{gather*}
   |\sigma|_{\mathrm{TV},\epsilon} = \int_\Omega |D\sigma|_\epsilon,
	\quad\text{and}\quad
	\|\cdot\|_\epsilon = \int_\Omega |\cdot|_\epsilon\,dx, \quad |\cdot|_\epsilon = \sqrt{|\cdot|^2 + \epsilon^2},
\end{gather*}
where $\epsilon>0$ is small and controls the degree of smoothing, and denote by $ J_{\sigma,\beta,\epsilon}(\kappa) $ the functional in \eqref{eq:linearized-functional} with the
smoothed norm $\|\cdot\|_\epsilon$ and seminorm $|\cdot|_{\rm TV,\epsilon}$:
\begin{align*}
    J_{\sigma,\beta,\epsilon}(\kappa) = \|H'(\sigma)\kappa - d_\sigma\|_{\epsilon} + \beta|\sigma+\kappa|_{\mathrm{TV},\epsilon}.
\end{align*}
The next result shows that a minimizer of $ J_{\sigma,\beta,\epsilon} $ converges to a minimizer of $ J_{\sigma,\beta} $ as $ \epsilon $ goes to zero.

\begin{theorem}
For any $\epsilon>0$, there exists at least one minimizer $\kappa_\epsilon$ to the functional
$J_{\sigma,\beta,\epsilon}$, and any accumulation point of the sequence of minimizers $\{\kappa_\epsilon\}_{\epsilon>0}$
is a minimizer of $J_{\sigma,\beta} $, as $\epsilon\to0^+$.
\end{theorem}
\begin{proof}
The existence of a minimizer to $J_{\sigma,\beta,\epsilon}$ follows easily as Theorem \ref{thm:minimizer}.
Clearly, $ J_{\sigma,\beta}(\kappa) \leq J_{\sigma,\beta,\epsilon}(\kappa) \leq J_{\sigma,\beta}(\kappa) + 2|\Omega|\epsilon $.
Let $ \kappa^\ast $ minimize $ J_{\sigma,\beta} $ and $ \kappa_\epsilon $ minimize $ J_{\sigma,\beta,\epsilon}$. Then by the
minimizing properties of $\kappa^*$ and $\kappa_\epsilon$ to $J_{\sigma,\beta}$ and $J_{\sigma,\beta,\epsilon}$, respectively, there hold
\begin{equation*}
 J_{\sigma,\beta}(\kappa^\ast) \leq J_{\sigma,\beta}(\kappa_\epsilon) \leq J_{\sigma,\beta,\epsilon}(\kappa_\epsilon) \leq J_{\sigma,\beta,\epsilon}(\kappa^\ast) \leq J_{\sigma,\beta}(\kappa^\ast) + 2|\Omega|\epsilon.
\end{equation*}
Thus any $ \{\kappa_{\epsilon_k}\}_{k\in\mathbb{N}} \subset \{\kappa_\epsilon\}_{\epsilon>0} $, $ \epsilon_k \to 0 $ as $ k \to \infty $, minimizes $ J_{\sigma,\beta}$, and by repeating the arguments in the proof of Theorem \ref{thm:minimizer}, we deduce that every convergent subsequence converges to a global minimzer.
\end{proof}

Next we derive a simple weighting scheme to facilitate the minimization of the functional $J_{\sigma,\beta,\epsilon}$.
Differentiating $J_{\sigma,\beta,\epsilon}$ with respect to $\kappa$ yields
\begin{align*}
    J_{\sigma,\beta,\epsilon}'(\kappa)\eta = \int_\Omega\frac{H'(\sigma)\kappa - d_\sigma}{|H'(\sigma)\kappa - d_\sigma|_\epsilon}H'(\sigma)\eta\,dx + \beta\int_\Omega \frac{\nabla(\sigma+\kappa)}{|\nabla(\sigma+\kappa)|_\epsilon}\cdot\nabla\eta\,dx.
\end{align*}
This is a highly nonlinear equation in $\kappa$ (or more precisely variational inequality, under the constraint $\kappa\in\mathcal{A}_\sigma$). Then we freeze denominators of the integrands at some $ \kappa'\in\mathcal{A}_\sigma$:
\begin{align*}
    J_{\sigma,\beta,\epsilon,\kappa'}'(\kappa)\eta = \int_\Omega\frac{H'(\sigma)\kappa - d_\sigma}{|H'(\sigma)\kappa' - d_\sigma|_\epsilon}H'(\sigma)\eta\,dx + \beta\int_\Omega \frac{\nabla(\sigma+\kappa)}{|\nabla(\sigma+\kappa')|_\epsilon}\cdot\nabla\eta\,dx.
\end{align*}
This corresponds to the gradient for the following weighted quadratic problem
\begin{align}
    J_{\sigma,\beta,\epsilon,\kappa'}(\kappa) 
    &= \frac{1}{2}\int_\Omega w(\sigma,\kappa')|H'(\sigma)\kappa - d_\sigma|^2\,dx + \frac{\beta}{2}\int w_0(\sigma,\kappa')|\nabla(\sigma+\kappa)|^2\,dx\nonumber \\
    &= \tfrac{1}{2}\|H'(\sigma)\kappa - d_\sigma\|_{L_{w}^2(\Omega)}^2 + \tfrac{\beta}{2}\|\nabla(\sigma+\kappa)\|_{L_{w_0}^2(\Omega)}^2,\label{eqn:fcnl-weighted}
\end{align}
with the weight functions $w(\sigma,\kappa')$ and $w_0(\sigma,\kappa)$ given by
\begin{equation} \label{eq:weights}
    w(\sigma,\kappa') = |H'(\sigma)\kappa' - d_\sigma|_\epsilon^{-1} \quad \text{and} \quad w_0(\sigma,\kappa') = |\nabla(\sigma+\kappa')|_\epsilon^{-1}.
\end{equation}
The reweighing step is repeated several times until a suitable stopping criterion is reached. This whole procedure for updating the increment $\kappa$ is in the same spirit of lagged diffusivity for total variation denoising \cite{Vogel1996} or iteratively reweighed least-squares. Numerically, it is fairly robust to the initialization. Clearly, the procedure easily generalizes to multiple data sets. 

\subsection{Numerical algorithm} \label{sec:alg-updt-appr}
To compute the update $\kappa$ from the functional $J_{\sigma,\beta,\epsilon,\kappa'}$, we employ the conjugate gradient method. To this end, we first derive the matrix representation. Let $ \{\phi_i\} $, $ \{\psi_i\}$ and $\{\xi_i\} $ be given bases for representing (discretizing)
$\kappa$, $\sigma$ and $d_\sigma$, respectively. Let
\begin{equation*}
    \kappa(x) = \sum_i \kappa_i\phi_i(x), \quad \sigma(x) = \sum_i\sigma_i\psi_i(x) \quad \mbox{and}\quad d_\sigma(x) = \sum_i(d_\sigma)_i\xi_i(x).
\end{equation*}
In our implementation, we take the standard P1 finite element basis as the basis for all three functions; and we
refer to Appendix \ref{sec:fem} for a convergence analysis of the finite element approximations for both nonlinear and linearized problems.
Next we define the corresponding stiffness and mass matrices respectively by
\begin{align*}
    (\mathbf K_{w_0})_{ij} = \int_\Omega w_0\nabla\psi_i\cdot\nabla \psi_j\,dx, \quad
    (\mathbf L)_{ij} = \int_\Omega wH'(\sigma)\phi_iH'(\sigma)\phi_j\,dx,\quad 
    \mbox{and}\quad (\mathbf U)_{ij} = \int_\Omega w\xi_iH(\sigma)\phi_j\,dx. 
\end{align*}
Note that $\mathbf{K}_{w_0}$ is sparse, but $ \mathbf L $ and $ \mathbf U $ are generally not sparse.
Then upon expansion and with the vector representation of the variables, problem \eqref{eqn:fcnl-weighted} reads
\begin{equation*}
    \min_{\boldsymbol\kappa}{\boldsymbol \kappa}^T(\mathbf L + \beta \mathbf{K}_{w_0})\boldsymbol{\kappa} - 2(\mathbf{d}_\sigma^T\mathbf{U} - \beta \boldsymbol\sigma^T\mathbf{K}_{w_0})\boldsymbol{\kappa} = \min_{\boldsymbol{\kappa}} \boldsymbol{\kappa}^T\mathbf{H}\boldsymbol{\kappa} - 2\mathbf{h}^T\boldsymbol{\kappa},
\end{equation*}
where $ \mathbf{H} = \mathbf{L} + \beta \mathbf{K}_{w_0} $ and $ \mathbf h = \mathbf{U}^T\mathbf
d_\sigma - \beta\mathbf{K}_{w_0}\boldsymbol{\sigma} $, and the optimality condition simply becomes
$ \mathbf{H}\boldsymbol\kappa - \mathbf{h} = 0 $. In the implementation, rather than forming  $\mathbf
H $ explicitly, it is beneficial to access $ \mathbf L $ only by matrix-vector product. To improve the numerical stability, one may add a small $ \delta\mathbf{I} $ to $ \mathbf{H} $, which allows to compensate the potential non-positivity of the matrix $\mathbf{H}$ due to numerical errors. 

Now we can present the detailed procedure in Algorithm \ref{alg:euclid}.
There are two stopping criteria for the algorithm: one for the iteratively
reweighed least-squares for updating the increment $\kappa$, and the other
for the linearization at the outer iteration. Either can be based on the
the relative change of the increment $\kappa$, and the second can be based
on the magnitude of the derivative of $ \mathcal J_\beta $ at $ \sigma $
in direction $\kappa' $. For the conjugate gradient iteration at Step 6,
we initialize the iteration with $\kappa'$, which can be fine tuned to be
$\lambda \kappa'$ for some $\lambda>0$, if needed, in a manner similar to
the heavy ball method. In our numerical experiments, this choice is a good warm starting strategy.

\begin{algorithm}
\caption{Numerical algorithm for Problem \eqref{eq:opt:problem}.}\label{alg:euclid}
\begin{algorithmic}[1]
\STATE Set initial guess $ \sigma_0 $, $\epsilon>0$ and maximum number of iterations $I$ and $K$. 
\FOR{$k=1,\ldots,K$}
\STATE Set $\kappa'=0$;
\FOR{$ i = 1,\dots,I $}
\STATE Compute weights $ w $ and $ w_0 $ by \eqref{eq:weights};
\STATE Solve \eqref{eqn:fcnl-weighted} by conjugate gradient method;
\STATE Check stopping criterion;
\STATE Set $ \kappa' = \kappa $;
\ENDFOR
\STATE Update $ \sigma_{k+1} = \sigma_k + \kappa' $;
\STATE Check stopping criterion;
\ENDFOR
\end{algorithmic}
\end{algorithm}

\section{Numerical results and discussions}\label{sec:numer}

Now we present numerical experiments to illustrate the proposed reconstruction algorithm.
\subsection{Experimental setting}
The algorithm is implemented
in Python using the \emph{DOLFIN 2017.2.0} (FEniCS) package 
\cite{Logg2012,LoggWells2010a}. Operators involving solving PDE's are formulated in the unified
form language of FEniCS and solved using standard solvers. The standard $ P1$ finite
elements are employed to approximate various functions. The meshes involved in generating
simulated data and reconstruction are obtained from the public software package \emph{gmsh 2.10.1}. All meshes are generated
from a circle shape, but with various refinement levels given by the characteristic length $ h $ (i.e., mesh size).
An overview of the statistics of the used
meshes is given in Table \ref{tbl:mesh}. In the experiments, we consider the following boundary data:
\begin{gather*} 
    f_1(x_1,x_2) = x_1, \quad
    f_2(x_1,x_2) = x_2, \quad
    f_3(x_1,x_2) = \tfrac{x_1+x_2}{\sqrt{2}}, \quad
    f_4(x_1,x_2) = \tfrac{x_1-x_2}{\sqrt{2}}.
\end{gather*}
The simulated power density data data $ \mathcal{H} = (H_1,H_2,H_3,H_4) $ corresponding to the boundary fluxes $ \mathcal{F} = (f_1,f_2,f_3,f_4) $ are
generated using a finer mesh $ M_1 $, and the reconstructions are performed on the coarser mesh $ M_2 $, in order to mitigate the so-called inverse crime. 

\begin{table}[htb!]
\centering
\caption{Mesh statistics.\label{tbl:mesh}}
    \begin{tabular}{lccc}
      \hline
        & Nodes & Triangles & $ h $ \\ \hline
        $ M_1 $ & 41690 & 84010 & 0.01 \\ \hline
        $ M_2 $ & 6523 & 13296 & 0.025 \\ \hline
    \end{tabular}
\end{table}

\begin{figure}[htb!]
    \centering
    \subfloat[Head model phantom\label{fig:headmodel}]{\includegraphics[width=0.25\textwidth]{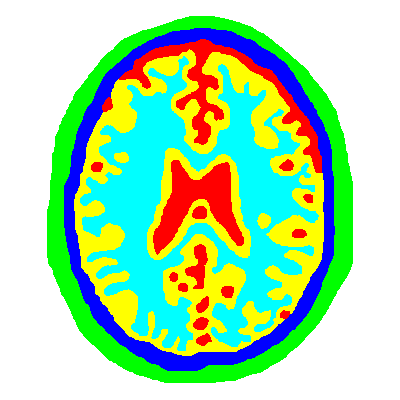}}
    \subfloat[Shape phantom\label{fig:shapes}]{\includegraphics[width=0.25\textwidth]{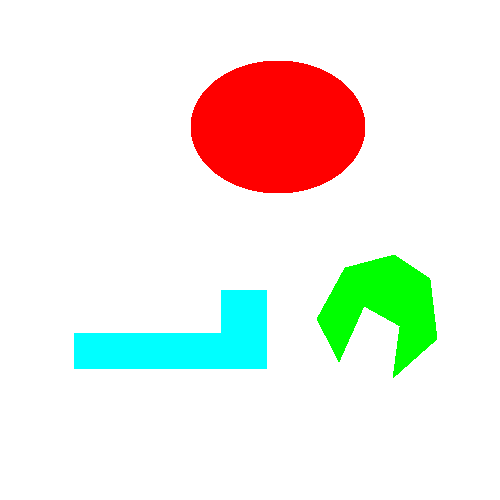}}
    \raisebox{1.451cm}{\subfloat[$\sigma$-values for different tissues.\label{tbl:tissues}]{\begin{tabular}{l|c|c}\hline
         Tissue/material & $ \sigma $ & Color  \\ \hline
         air & 0.4 & white \\ \hline
         scalp & 0.5232 & green \\ \hline
         skull & 0.2983 & blue \\ \hline
         spinal fluid & 1.0143 & red \\ \hline
         gray matter & 0.55946 & yellow \\ \hline
         white matter & 0.32404 & cyan \\ \hline
    \end{tabular}}}
    
    \subfloat[Meshed head model  \label{fig:head_true_sigma}]{\includegraphics[width=.225\textwidth]{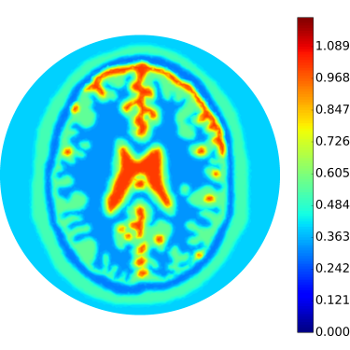}}~
    \subfloat[Meshed shape \label{fig:shape_true_sigma}]{\includegraphics[width=0.225\textwidth]{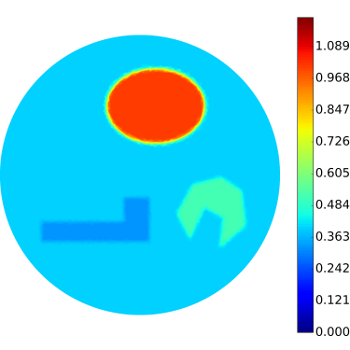}}
    
    \caption{The two phantoms: (a) brain phantom and (b) geometric shapes, for simulated data and reconstruction, and (c) table of color-tissue correspondence.}
    \label{fig:phantom}
\end{figure}

We consider two different phantoms: the brain phantom in Fig. \ref{fig:headmodel}\cite{LiKaramehmedovicSherinaKnudsen2018}
and the geometrical shapes phantom in Fig. \ref{fig:shapes}. The corresponding interior data are shown in Fig. \ref{fig:Hdata}.
When implementing the algorithm, the maximum number $I$ of inner iterations is fixed at $I=3$, and the number of conjugate gradient iterations is also fixed at $3$. Our experiments indicate that increasing these numbers does not improve much the reconstruction quality. 
Throughout, the regularization parameter $\beta$ is determined in a trial-and-error way. The noisy data $ \tilde{\mathbf{H}} $ is generated componentwise according to
$$ \tilde{\mathbf{H}} = \mathbf H + \delta_e\tfrac{|\mathbf{H}|}{|\mathbf{e}|}\mathbf{e} $$
where $\mathbf{e}$ is a vector of appropriate size with each entry following a standard Gaussian distribution, $\delta_e$ the relative noise level, and $ |\cdot| $ denotes the Euclidean norm of a vector.

\begin{figure}
   \subfloat[$ H_1 $ data\label{fig:head_H1}]{\includegraphics[width=0.225\textwidth]{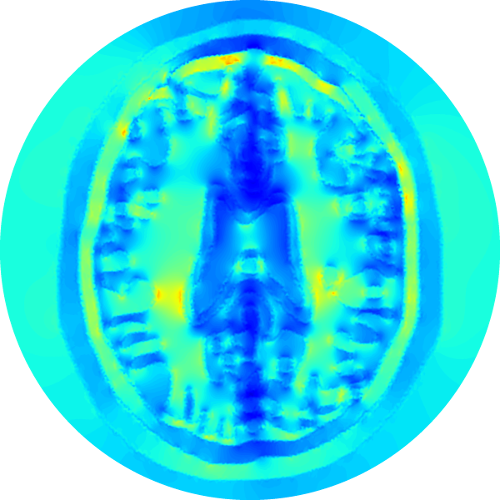}}~
   \subfloat[$H_2$ data \label{fig:head_H2}]{\includegraphics[width=0.225\textwidth]{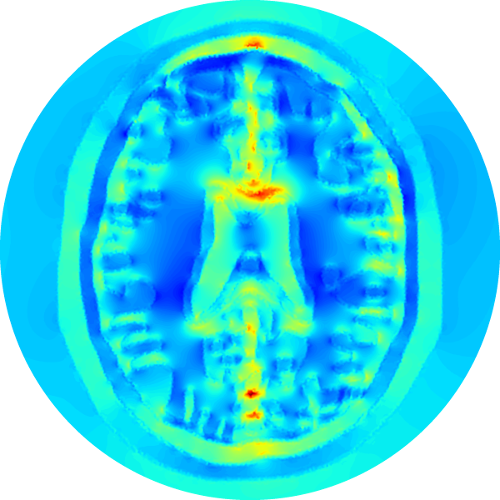}}~
   \subfloat[$ H_3 $ data\label{fig:head_H3}]{\includegraphics[width=0.225\textwidth]{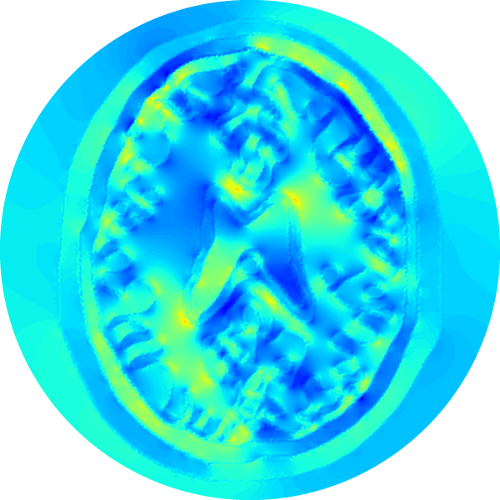}}~
   \subfloat[$ H_4 $ data\label{fig:head_H4}]{\includegraphics[width=0.225\textwidth]{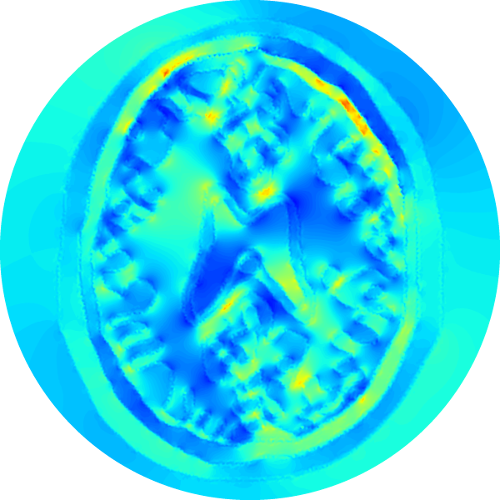}}~
   \subfloat{\includegraphics[width=0.0365\textwidth]{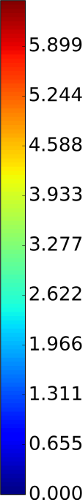}}
   
   \addtocounter{subfigure}{-1}
   \subfloat[$ H_1 $ data\label{fig:shape_H1}]{\includegraphics[width=0.225\textwidth]{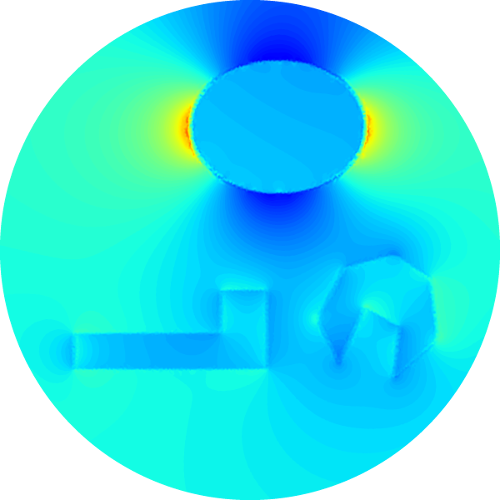}}~
   \subfloat[$H_2$ data \label{fig:shape_H2}]{\includegraphics[width=0.225\textwidth]{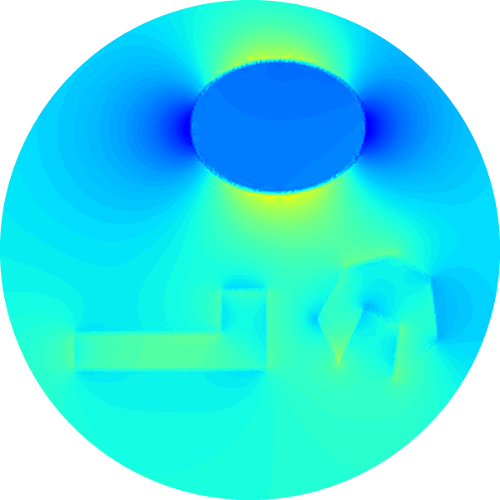}}~
   \subfloat[$ H_3 $ data\label{fig:shape_H3}]{\includegraphics[width=0.225\textwidth]{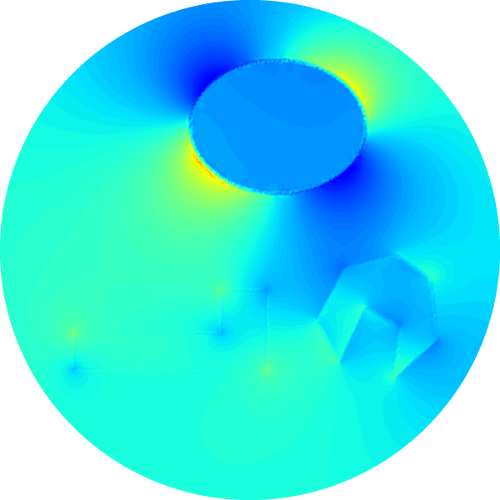}}~
   \subfloat[$ H_4 $ data\label{fig:shape_H4}]{\includegraphics[width=0.225\textwidth]{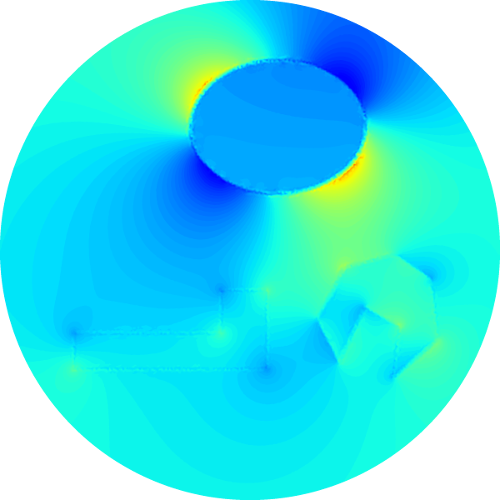}}~
   \subfloat{\includegraphics[width=0.0365\textwidth]{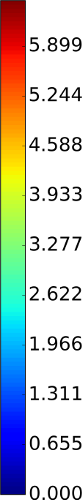}}
   
   \caption{The power density data corresponding to $ f_i $, $ i = 1,\ldots,4$ for each of the two phantoms.}       \label{fig:Hdata}
\end{figure}

\subsection{Convergence of the algorithm}

To gain insights into Algorithm \ref{alg:euclid}, we present some numerical results on the convergence behavior in Fig. \ref{fig:conv}. In Figs. \ref{fig:conv_func} and \ref{fig:conv_func_n1}, we show the evolution of the functional value $ \mathcal{J}_\beta(\sigma_k) $ during the first 50 outer iterations, where the thick dashed line denotes the functional $ \mathcal{J}_\beta(\sigma^\ast)$ evaluated at the true conductivity $\sigma^\ast$ (interpolated at the reconstruction mesh). 
Due to regularization and discretization, it is not surprising that the true conductivity $\sigma^*$ is generally not a global minimizer to $\mathcal{J}_\beta$. 
It is observed that the functional value decreases steadily as the iteration proceeds for both exact and noisy data, which shows clearly the robustness of the algorithm, and for exact data, eventually, the functional value falls below $\mathcal{J}_\beta(\sigma^*)$.

In Figs. \ref{fig:conv_err} and \ref{fig:conv_err_n1}, we show the reconstruction errors in several metrics along the iteration. 
Given the choice of the space $ BV(\Omega)$, we use the $ L^1(\Omega)$-norm and total variation difference and a metric related to the intermediate topology of $ BV(\Omega) $:
\begin{equation*} d_{BV}(\sigma,\eta) = \|\sigma-\eta\|_{L^1(\Omega)} + \left||\sigma|_\text{TV} - |\eta|_\text{TV}\right|.
\end{equation*}
The peculiar jump in the plots at the beginning is related to the difference in total variation. This might be attributed to the fact that two functions with identical total variation can look anything alike. Thus it is mostly the tail of the plot, where the $ L^1(\Omega)$-difference gets small, which is of significance. The error in total variation is dominating when compared with the $L^1(\Omega)$ error. The plot shows clearly that the convergence of the algorithm is quite steady for both exact and noisy data.

\begin{figure}[!htb]
    \centering
    \subfloat[functional value $ \mathcal{J}_\beta(\sigma_k) $\label{fig:conv_func}]{\includegraphics[width=0.49\textwidth]{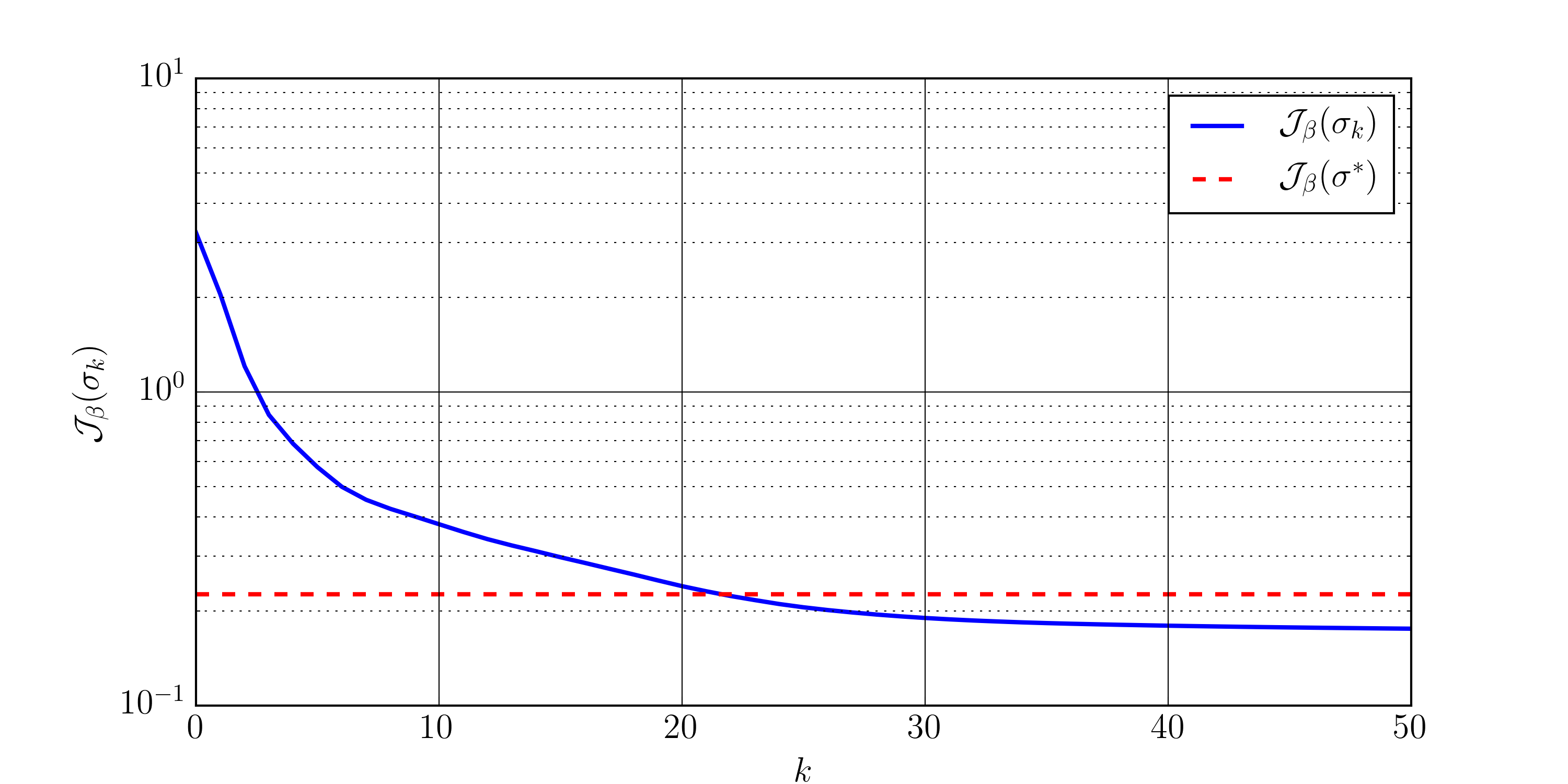}}
    \subfloat[errors\label{fig:conv_err}]{\includegraphics[width=0.49\textwidth]{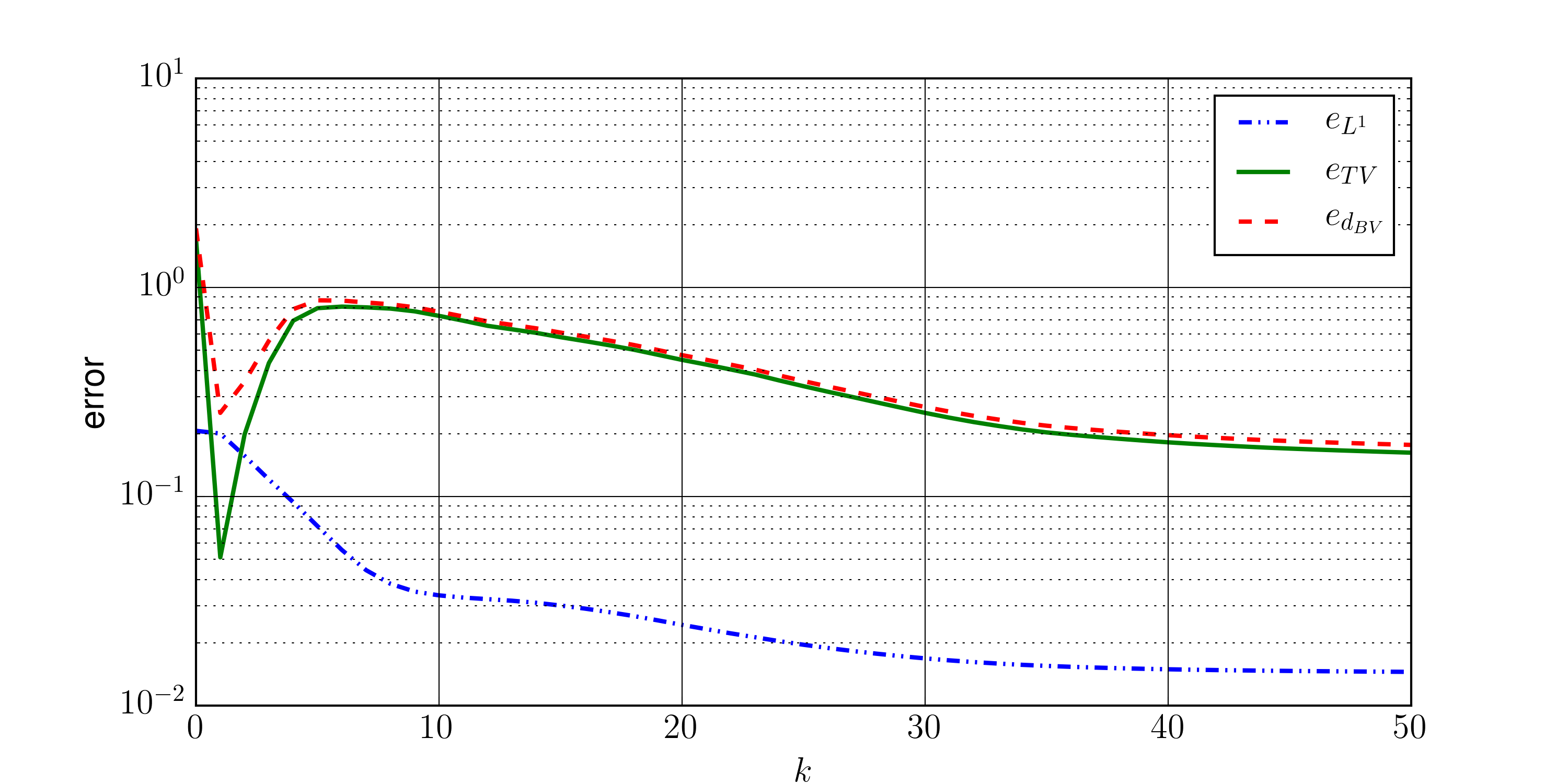}}\\
    \subfloat[functional value $ \mathcal{J}_\beta(\sigma_k) $\label{fig:conv_func_n1}]{\includegraphics[width=0.49\textwidth]{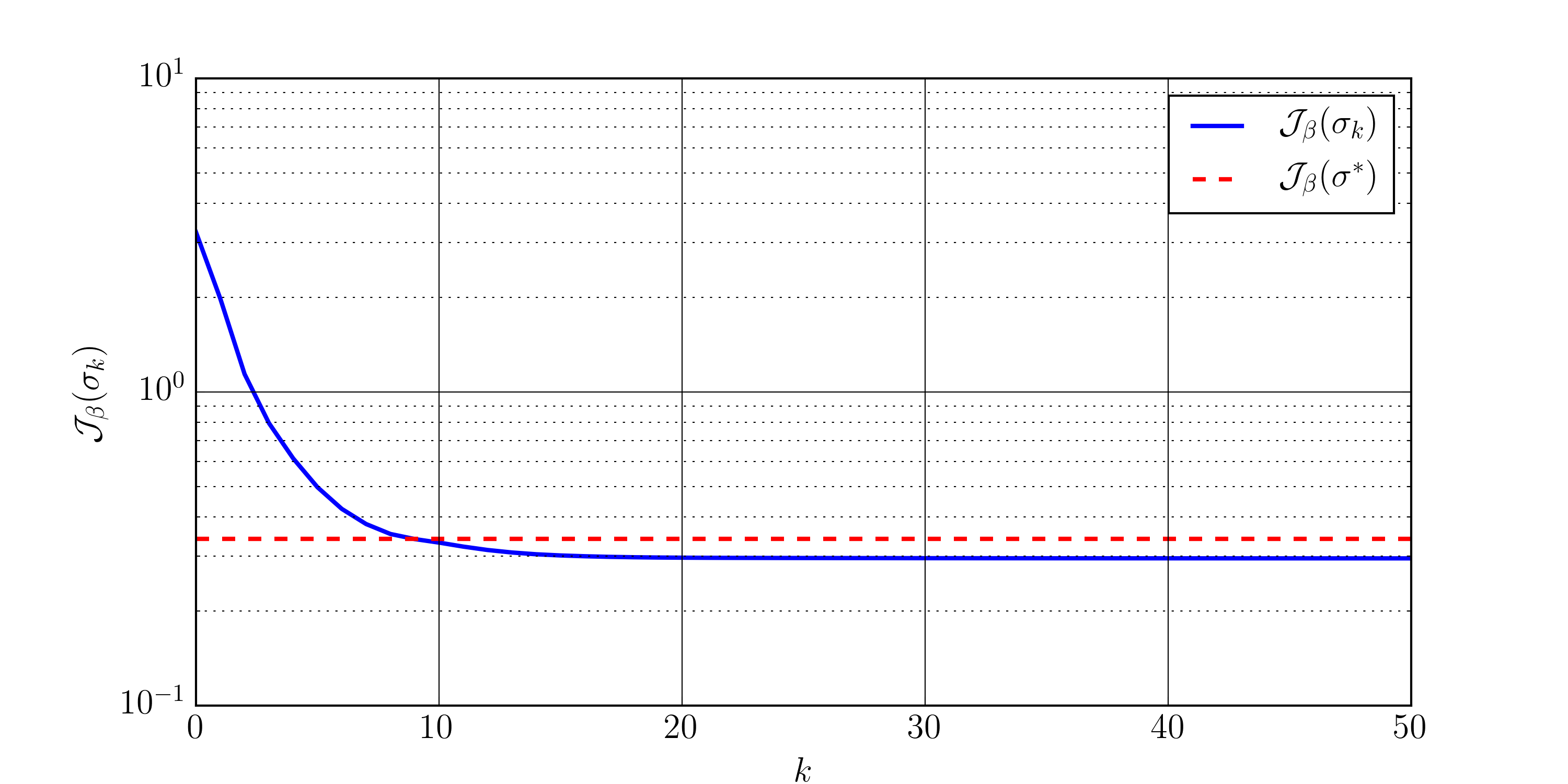}}
    \subfloat[errors\label{fig:conv_err_n1}]{\includegraphics[width=0.49\textwidth]{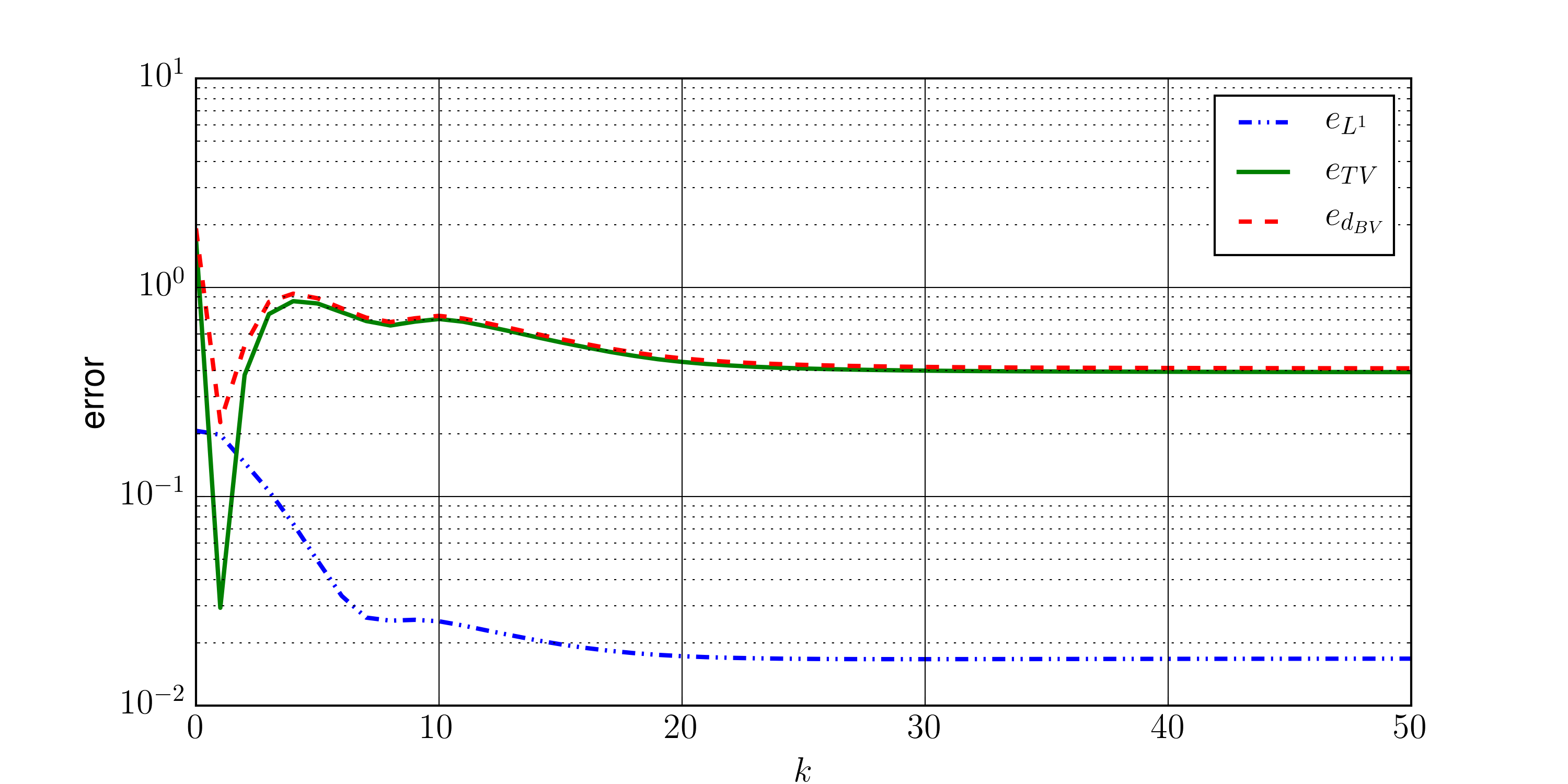}}
    \caption{Convergence plots for the Shape phantom in Fig. \ref{fig:shapes}  with exact data (top) and 1\% noise (bottom). The parameters are taken to be $ \beta = 3.5\times 10^{-2} $ and $ \epsilon = 10^{-4} $, $ \lambda = 0.8 $ and the data are $ H_1,H_2,H_3 $. The red line in (a) and (c) refers to $\mathcal{J}_\beta(\sigma^*)$. Figs. (b) and (d) show three metrics of the error $ e_{L^1} := \|\sigma_k - \sigma^\ast\|_{L^1(\Omega)} $, $ e_{TV} = ||\sigma_k|_\text{TV} - |\sigma^\ast|_{\text{TV}}| $ and $ e_{d_{BV}} = d_{BV}(\sigma_k,\sigma^\ast) $. \label{fig:conv}}
\end{figure}

\subsection{Numerical reconstructions for full data}
The reconstructions at two noise levels and different combinations of interior data are presented in Fig.
\ref{fig:recons_cg}, where the parameters for the reconstructions are shown in the captions. It is
observed when the noise level increases, some fine details disappear in the reconstructions. 
Note that some details, e.g., the upper left and right part of the skull in Fig. \ref{fig:head2_rec_noise_5},
can still be recovered by imposing less regularization, but at the cost of sacrificing the accuracy at the
remaining part. The background in Fig. \ref{fig:shape3_rec_noise_1} appears slightly noisy, which is actually due to the discretization
of the color-spectrum: 100 hues are used in the plots, but by changing it to 99 or 101, it is nearly completely uniform.

For the cases with only two boundary measurements, certain directional features are favored by the specified boundary data. The pairs $(f_1,f_2)$ and $(f_3,f_4)$ yields artifact that are different
from each other, showing the influence of the choice of boundary data. For a microlocal analysis of the artifacts in the linearized model of AET, we refer interested readers to the recent work \cite{BalKnudsen:2018}.

\begin{figure}[htb!]
    \centering
    \subfloat[1\% noise, $ H_1,H_2,H_3 $\label{fig:head3_rec_noise_1}]{\includegraphics[width=.225\textwidth]{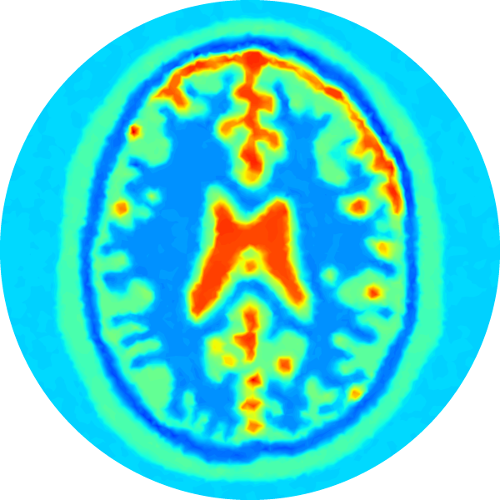}}~
    \subfloat[5\% noise, $ H_1,H_2,H_3 $\label{fig:head3_rec_noise_5}]{\includegraphics[width=.225\textwidth]{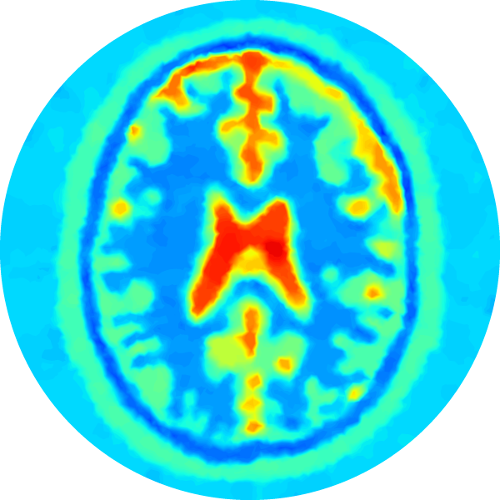}}~
    \subfloat[1\% noise, $ H_1,H_2,H_3 $\label{fig:shape3_rec_noise_1}]{\includegraphics[width=.225\textwidth]{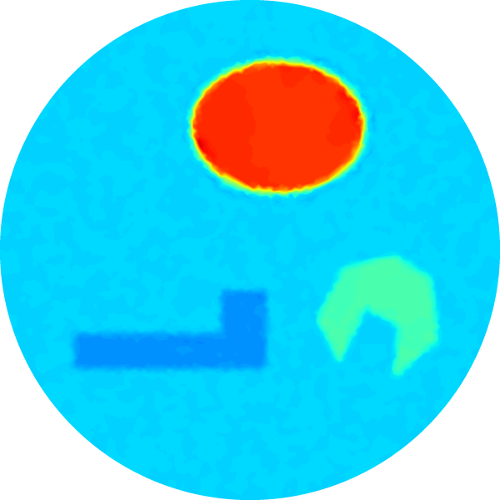}}~
    \subfloat[5\% noise, $ H_1,H_2,H_3 $\label{fig:shape3_rec_noise_5}]{\includegraphics[width=.225\textwidth]{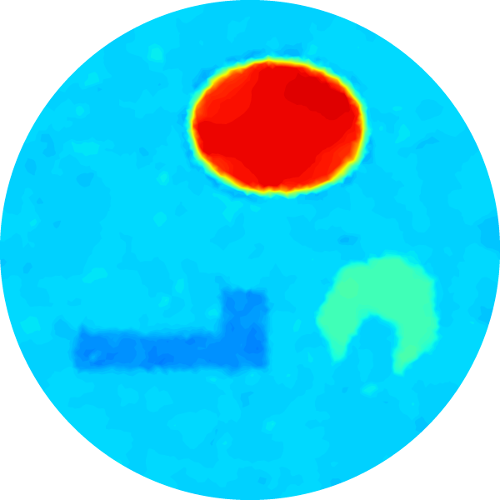}}~
    \subfloat{\includegraphics[width=.035\textwidth]{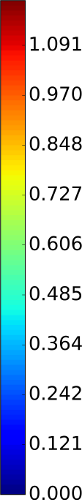}}

    \addtocounter{subfigure}{-1}
    \subfloat[1\% noise, $ H_1,H_2$\label{fig:head2_rec_noise_1}]{\includegraphics[width=.225\textwidth]{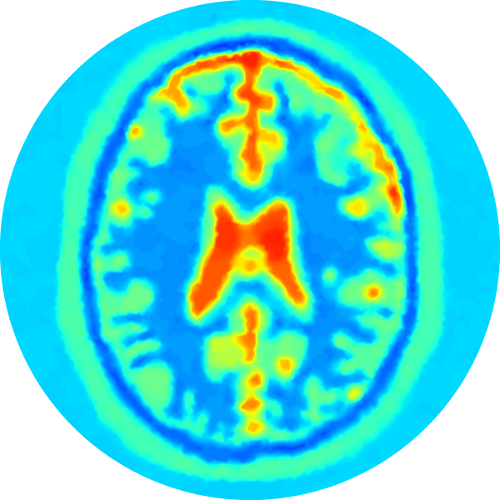}}~
    \subfloat[5\% noise, $ H_1,H_2$\label{fig:head2_rec_noise_5}]{\includegraphics[width=.225\textwidth]{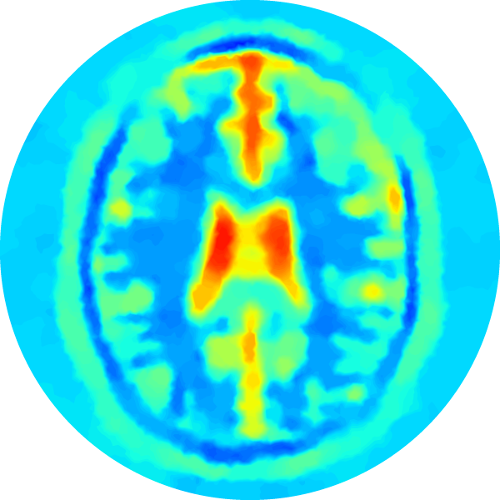}}~
    \subfloat[1\% noise, $ H_1,H_2$\label{fig:shape2_rec_noise_1}]{\includegraphics[width=.225\textwidth]{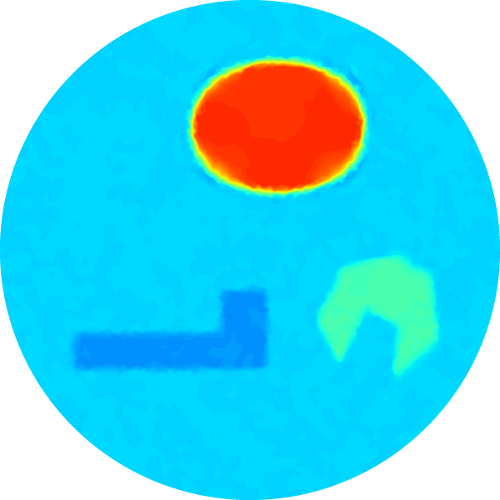}}~
    \subfloat[5\% noise, $ H_1,H_2$\label{fig:shape2_rec_noise_5}]{\includegraphics[width=.225\textwidth]{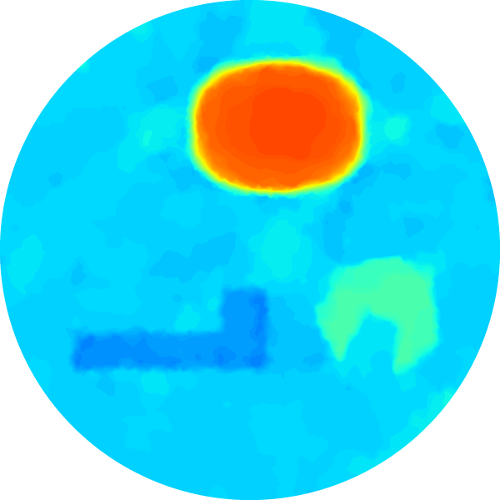}}~
    \subfloat{\includegraphics[width=.035\textwidth]{gfx_cg/cbar.png}}

    \addtocounter{subfigure}{-1}
    \subfloat[1\% noise, $ H_3,H_4$\label{fig:head2diag_rec_noise_1}]{\includegraphics[width=.225\textwidth]{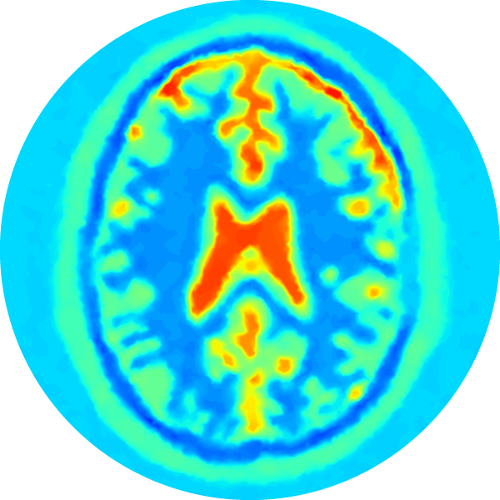}}~
    \subfloat[5\% noise, $ H_3,H_4$\label{fig:head2diag_rec_noise_5}]{\includegraphics[width=.225\textwidth]{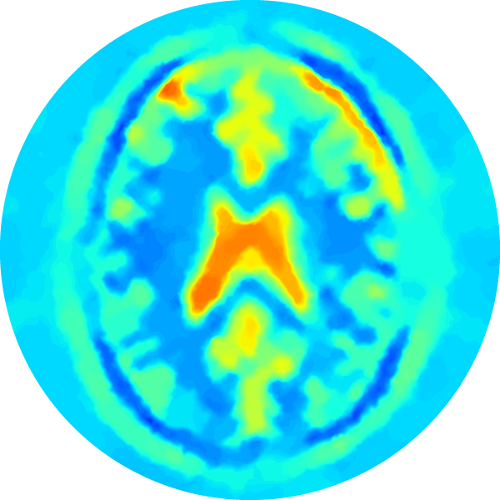}}~
    \subfloat[1\% noise, $ H_3,H_4$\label{fig:shape2diag_rec_noise_1}]{\includegraphics[width=.225\textwidth]{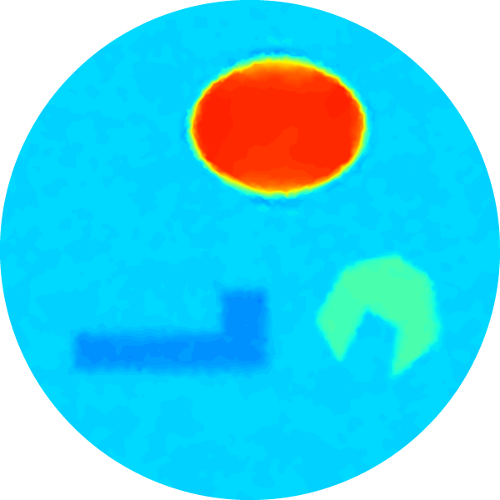}}~
    \subfloat[5\% noise, $ H_3,H_4$\label{fig:shape2diag_rec_noise_5}]{\includegraphics[width=.225\textwidth]{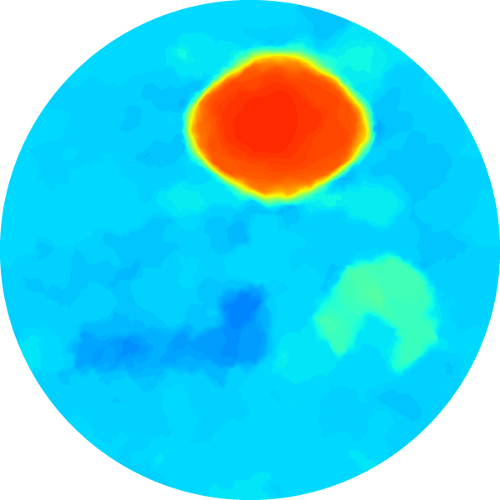}}~
    \subfloat{\includegraphics[width=.035\textwidth]{gfx_cg/cbar.png}}

    \caption{Reconstructions with different levels of noise, different amounts of data and different phantoms.  Parameters: $ \lambda = 0.8 $, $ \epsilon = 10^{-4} $, $ \beta = 3.5\times 10^{-2} $ for the 1\% noise cases. $ \beta = 0.7 $ for the 5\% noise cases. Colorbar scale is $ (0, 1.2) $.}
    \label{fig:recons_cg}
\end{figure}

\subsection{Numerical reconstructions for partial data}

The variational formulation in Section \ref{sec:rec} extends straightforwardly to partial data, by restricting the integral in the fidelity term to a subdomain and with obvious modification, Algorithm \ref{alg:euclid} extends easily. To illustrate this flexibility, we consider two different subdomains for the available interior power density data: one small concentric disc and one half-disc. 
The corresponding numerical results are presented in Fig. \ref{fig:recons_cg_partial}, where
the dashed red lines mark the boundary between domains with data and without. Notably, the reconstructions within the data-domains are fairly accurate, whereas the exterior has no significant updates during the iteration, and it is dominated completely by the initial guess $ \sigma_0 $.

Note that the scale of the colorbar in these plots is slightly larger compared to the other plots. This is due to an interesting effect happening near the boundary of the data-domain. Note that the inclusions in the phantom are reconstructed quite accurately near the center of the data-domain. However, as the boundary is approached, the values deviate from the background, seemingly in an attempt to compensate for the missing conductivity reconstruction outside the data-domain. However, the precise mechanism for the phenomenon is to be ascertained.

\begin{figure}[htb!]
    \centering
    \subfloat[Subdomain: inner disk        \label{fig:shape3_partial401}]{\includegraphics[width=.325\textwidth]{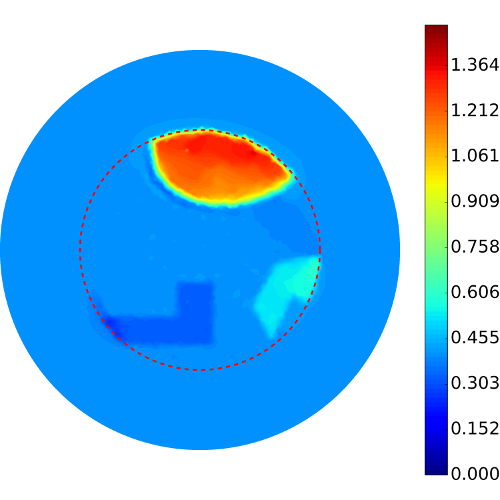}}~
    \subfloat[Subdomain: right half-disc.\label{fig:shape3_partial402}]{\includegraphics[width=.325\textwidth]{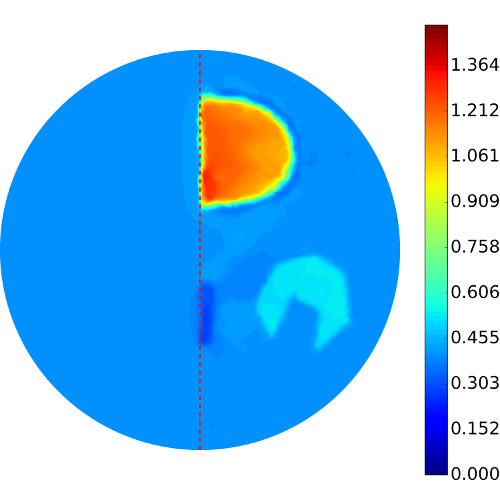}}

    \caption{The reconstructions with data available on subdomains (the boundary is indicated by the dashed curves), obtained with 1\% noise and three densities $ H_1 $, $ H_2 $ and $ H_3 $ data. Parameters: $ \lambda = 0.8 $, $ \epsilon = 10^{-4} $, $ \beta = 3.5\times 10^{-2} $. Colorbar scale is $ (0, 1.5) $.}
    \label{fig:recons_cg_partial}
\end{figure}

\section{Conclusion}
In this paper we have studied the numerical reconstruction of acousto-electric tomography under very weak regularity assumptions on the conductivity. We have shown various continuity and differentiability results on the solutions of the elliptic PDE and the parameter-to-data map. We proposed a reconstruction algorithm based on total variation penalty, and showed the existence of a minimizer and its stability. Further, we proposed an algorithm based on recursively linearizing the forward map, smoothing and lagged diffusivity approximation, together with the conjugate gradient method for solving the resulting variational formulation. We have demonstrated the accuracy of the approach with extensive numerical experiments with full and partial data.

\appendix
\section{Convergence of finite element approximations}\label{sec:fem}
In this appendix, we discuss the convergence of the finite element approximation of the
functional $\mathcal{J}_\beta$ and its linearization $J_{\beta,\sigma}$. Let $\mathcal{T}$ be a
quasi-uniform partition of the domain $\Omega$ into simplicial elements, and consider
the standard conforming piecewise linear finite element space:
\begin{equation*}
  X_h = \{\chi\in C(\overline\Omega):\; \chi|_T\in\mathcal{P}_1\ \forall T\in \mathcal{T}_h\},
\end{equation*}
where $\mathcal{P}_1(T)$ denotes the space of linear functions on $ T$, and $V_h:= X_h\cap V$. The
space $V_h$ is used to discretize the state and adjoint variables, and $X_h$ the conductivity $\sigma$. We
denote by $I_h:C(\overline{\Omega})\to X_h$ the standard Lagrangian nodal interpolation operator,
and $R_h:V\to V_h$ the Ritz projection:
\begin{equation*}
    (\nabla R_hv,\nabla \phi)= (\nabla v,\nabla\phi)\quad \forall v\in V,\phi\in V_h.
\end{equation*}
The operators satisfy the following approximation properties:
\begin{equation}\label{eqn:approx}
    \begin{aligned}
       \lim_{h\to0}\|v-I_hv\|_{W^{1,p}(\Omega)}&=0\quad \forall v\in W^{1,p}(\Omega),\ \ p>d,\\
       \lim_{h\to 0}\|v-R_hv\|_{H^1(\Omega)}&=0\quad \forall v\in V.
    \end{aligned}
\end{equation}
Below we  discuss the discretization of the functionals $\mathcal{J}_\beta$ and $J_{\beta,\sigma}$ separately. Throughout the appendix, we assume and $f\in (W^{1-\frac{1}{r},r}(\Gamma))'$, $r>2$.

First we discuss the functional $\mathcal{J}_\beta$. For any $\sigma_h\in \mathcal{A}_h$, the discrete forward problem is to find $u_h\equiv u_h(\sigma_h)\in V_h$ such that
\begin{equation}\label{eqn:fem}
  (\sigma_h\nabla u_h,\nabla \chi) = (f,\chi)_{L^2(\Gamma)} \quad \forall \chi\in V_h.
\end{equation}
Then with $H_h(\sigma_h)=\sigma_h|\nabla u_h(\sigma_h)|^2$,
the discrete analogue $\mathcal{J}_{\beta,h}$ of the functional $\mathcal{J}_\beta$ is given by
\begin{equation*}
    \mathcal{J}_{\beta,h}(\sigma_h) = \|H_h(\sigma_h)-z\|_{L^1(\Omega)} + \beta|\sigma_h|_{\rm TV},
\end{equation*}
which is to be minimized over the discrete admissible set $\mathcal{A}_h=\mathcal{A}\cap V_h$.
It is easy to obtain the existence of a minimizer $\sigma_h^*\in\mathcal{A}_h$.

The following discrete analogue of Theorem \ref{thm:reg} is useful. 
\begin{lemma}\label{lem:dis-Meyers}
For any $\sigma_h\in \mathcal{A}_h$, 
there exist some $q\in (2,r)$ and some constant $C$ independent $h$
such that the solution $u_h(\sigma_h)$ to problem \eqref{eqn:fem} satisfies
\begin{equation*}
  \|u_h(\sigma_h)\|_{W^{1,q}(\Omega)} \leq C\|f\|_{(W^{1-\frac1r,r}(\Gamma))'}.
\end{equation*}
\end{lemma}
\begin{proof}
By the repeating the argument of \cite[Section 8.6]{BrennerScott:2008}
(the proof of Proposition 8.6.2), there exist
some $C>0$ and $q>2$ sufficiently close to $2$ such that for all $u_h\in V_h$, there holds
\begin{equation}\label{eqn:inf-sup}
\|u_h\|_{W^{1,q}(\Omega)} \leq C\sup_{0\neq v_h\in V_h}\frac{(\sigma_h\nabla u_h,\nabla v_h)}{\|v_h\|_{W^{1,q'}(\Omega)}}.
\end{equation}
Taking $u_h=u_h(\sigma_h)$ in the inequality, \eqref{eqn:fem}
and H\"{o}lder's inequality yield the desired assertion.
\end{proof}

We need a discrete analogue of Lemma \ref{lem:H-cont} on
the discrete forward map $\sigma_h\mapsto H_h(\sigma_h)$.
\begin{lemma}\label{lem:polyhed}
Let the sequence $\{\sigma_h\}_{h>0}\subset \mathcal{A}_h\subset\mathcal{A}$ converge
in $L^r(\Omega), \ r\geq 1$, to some $\sigma\in\mathcal{A}$ as $h$ tends to zero.
Then $H_h(\sigma_h) $ converges
to $H(\sigma)$ in $L^1(\Omega)$ as $h\to0^+$.
\end{lemma}
\begin{proof}
By Lax-Milgram theorem, $u$ and $u_h$ are uniformly bounded in $H^1(\Omega)$ independent of $h$. Setting the test
function $\chi=R_hu-u_h\in V_h\subset V$ in the weak formulations and then subtracting them give
\begin{equation*}
   \begin{aligned}
      \int_\Omega\sigma_h|\nabla(u-u_h)|^2\,dx=&-\int_\Omega(\sigma-\sigma_h)\nabla u\cdot\nabla(R_hu-u_h)\,dx \\
       &+
       \int_\Omega\sigma_h\nabla(u-u_h)\cdot\nabla(u-R_h u)\,dx:= {\rm I} + {\rm II}.
   \end{aligned}
\end{equation*}
It suffices to estimate the two terms $ \rm I$ and $ \rm II$. For
the term $ \rm I$, H\"{o}lder's inequality gives
\begin{equation*}
 |{\rm I}|\leq \|\sigma-\sigma_h\|_{L^p(\Omega)}\|\nabla u\|_{L^q(\Omega)}\|\nabla (R_hu-u_h)\|_{L^2(\Omega)},
\end{equation*}
where the exponent $q>2$ is from Theorem \ref{thm:reg}, and $p^{-1}+q^{-1}=2^{-1}$. Further, we note
\begin{equation*}
   \begin{aligned}
      \|\nabla(R_hu-u_h)\|_{L^2(\Omega)}
      &\leq C(\|u\|_{H^1(\Omega)}+\|u_h\|_{H^1(\Omega)})\leq C.
   \end{aligned}
\end{equation*}
Thus, $ \rm I \rightarrow0$ as $h\to0^+$, by the bound on  $\sigma_h$ and the convergence $\sigma_h\to \sigma^*$ in  $L^r(\Omega)$. Further, by the bound on $\sigma_h$,
\begin{equation*}
   \begin{aligned}
      |{\rm II}| &\leq \|\sigma_h\|_{L^\infty(\Omega)}\|\nabla(u-u_h)\|_{L^2(\Omega)}\|\nabla(u-R_hu)\|_{L^2(\Omega)}\\
         & \leq \lambda^{-1}\|\nabla(u-u_h)\|_{L^2(\Omega)}\|\nabla(u-{R}_hu)\|_{L^2(\Omega)},
   \end{aligned}
\end{equation*}
which tends to zero, in view of \eqref{eqn:approx}. These two estimates imply $u_h\to u$ in $H^1(\Omega)$.
Hence,
\begin{equation*}
  \begin{aligned}
    \|H_h(\sigma_h) - H(\sigma)\|_{L^1(\Omega)} &= \int_\Omega |\sigma_h|\nabla u_h(\sigma_h)|^2 - \sigma|\nabla u(\sigma)|^2|\,dx\\
    & \leq\int_\Omega |\sigma_h-\sigma|\,|\nabla u_h(\sigma_h)|^2\,dx +\int_\Omega\sigma|(|\nabla u_h(\sigma_h)|^2-|\nabla u(\sigma)|^2)|\,dx:={\rm III}+{\rm IV}
    \end{aligned}
\end{equation*}
By Lemma \ref{lem:dis-Meyers} and the $L^\infty(\Omega)$ bound on $\mathcal{A}_h$, $\lim_{h\to0^+}{\rm III} =0$.
The term $\rm IV$ also tends to zero due to $u_h\to u$ in $H^1(\Omega)$. This completes the proof.
\end{proof}

The next result gives the convergence of the discrete minimizers $\sigma_h^*$.
\begin{theorem}\label{thm:conv-fem}
The sequence $ \{\sigma_h^*\in X_h\}_{h>0}$ of minimizers to the discrete functionals $\mathcal{J}_{\beta,h}(\sigma_h)$ contains a subsequence converging in $L^1(\Omega)$ to a minimizer of $\mathcal{J}_\beta(\sigma)$ as $h\to0^+$.
\end{theorem}
\begin{proof}
Since $\sigma_h\equiv1\in\mathcal{A}_h$ for all $h$, the minimizing property of
$\sigma_h^\ast$ shows that the sequence $\{\mathcal{J}_{\beta,h}(\sigma_h^\ast)\}$ is uniformly
bounded. Thus $\{|\sigma_h^\ast|_{\rm TV}\}$ is uniformly bounded, and there exists a
subsequence, again denoted by $\{\sigma_h^\ast\}$, and some $\sigma^\ast \in\mathcal{A}$, such
that $\sigma_h^\ast\rightarrow \sigma^\ast$ weak $\ast$ in $\mathrm{BV}(\Omega)$. By Lemma \ref{lem:BV-props}(i),
$\sigma_h^\ast\rightarrow \sigma^\ast$ in $L^1(\Omega)$. Lemma \ref{lem:BV-props}(ii) implies
$|\sigma^\ast|_{\rm TV}\leq \liminf_{h\rightarrow0}|\sigma_h^\ast|_{\rm TV}.$ This and Lemma \ref{lem:polyhed}
imply
\begin{equation}\label{eqn:lsc}
   \mathcal{J}_\beta(\sigma^*) \leq \liminf_{h\to0^+} \mathcal{J}_{\beta,h}(\sigma_h^*).
\end{equation}
For any $\sigma\in\mathcal{A}$, Lemma \ref{lem:BV-props}(iii) implies the existence of a sequence
$\{\sigma^\epsilon\}\subset C^\infty(\overline{\Omega})$ such that
$\int_\Omega |\sigma^\epsilon-\sigma|dx<\epsilon$ and $\left|\int_\Omega|\nabla \sigma^\epsilon|dx
-\int_\Omega|D\sigma|\right|<\epsilon$. Let $\tilde{\sigma}_\epsilon =P_{[\lambda,\lambda^{-1}]}
\sigma^\epsilon$, where $P_{[\lambda,\lambda^{-1}]}$ denotes pointwise projection.
Since $\nabla \tilde{\sigma}^\epsilon =\nabla \sigma^\epsilon\chi_{\Omega_\epsilon}$ (with
the set $\Omega_\epsilon = \{x\in\Omega: \lambda\leq \sigma^\epsilon\leq\lambda^{-1}\}$), which is
uniformly bounded, and thus $\tilde{\sigma}^\epsilon\in \mathcal{A}\cap W^{1,\infty}(\Omega)$.
The minimizing property of
$\sigma_h^\ast\in\mathcal{A}_h$ gives $\mathcal{J}_h(\sigma_h^\ast)\leq \mathcal{J}_{\beta,h}(\mathcal{I}_h\tilde{\sigma}^\epsilon)$
for any $\epsilon>0$. In view of \eqref{eqn:approx}, since $\tilde{\sigma}^\epsilon\in W^{1,\infty}(\Omega)$,
we deduce
$\lim_{h\rightarrow0^+}\mathcal{I}_h\tilde{\sigma}^\epsilon = \tilde{\sigma}^\epsilon$ in $ W^{1,1}(\Omega)$.
Letting $h$ to zero, and  Lemma \ref{lem:polyhed} and
\eqref{eqn:lsc} yield $\mathcal{J}_\beta(\sigma^\ast)\leq \mathcal{J}_\beta(\tilde{\sigma}^\epsilon)$. Then, by the contraction property of $P_{[\lambda,\lambda^{-1}]}$, there hold
\begin{equation*}
  \begin{aligned}
   \int_\Omega |\nabla\tilde{\sigma }^\epsilon|\,dx &= \int_{\Omega_\epsilon}|\nabla\sigma^\epsilon|\,dx \leq \int_\Omega
    |\nabla \sigma^\epsilon|\,dx \leq \int_\Omega |D\sigma|+\epsilon,\\
    \int_\Omega |\tilde{\sigma}^\epsilon-\sigma|\,dx &\leq \int_\Omega|\sigma^\epsilon-\sigma|\,dx<\epsilon.
  \end{aligned}
\end{equation*}
Letting $\epsilon$ to zero and Lemma \ref{lem:H-cont}
 imply $\mathcal{J}_\beta(\sigma^\ast)\leq \mathcal{J}_\beta(\sigma)$ for any
$\sigma\in\mathcal{A}$, completing the proof.
\end{proof}

Next we discuss the discretization of the linearized problem \eqref{eq:linearized-functional} at some fixed $\sigma_0
\in \mathcal{A}$. With the approximation $u_h\equiv u_h(\sigma_0)$ defined by
\begin{equation*}
  (\sigma_0\nabla u_h,\nabla \chi) = (f,\chi)_{L^2(\Gamma)}\quad \forall \chi\in V_h.
\end{equation*}
Then, for any $\kappa_h\in \mathcal{A}_{\sigma,h}$, find $v_h\equiv u_h'(\sigma_0)[\kappa_h]\in V_h$ such that
\begin{equation*}
  (\sigma_0\nabla v_h,\nabla \chi) = - (\kappa_h\nabla u_h,\nabla \chi)\quad \forall \chi\in V_h.
\end{equation*}
Last, the discrete linearized functional ${J}_{\beta,h}$ of
$J_{\beta}$ (by omitting the subscript $\sigma_0$) reads
\begin{equation*}
   J_{\beta,h}(\kappa_h) = \|H_h'(\sigma_0)[\kappa_h]+H_h(\sigma_0)-z\|_{L^1(\Omega)} + \beta|\sigma_0+\kappa_h|_{\rm TV},
\end{equation*}
where the linearized parameter-to-data map $H_h'(\sigma_0)[\kappa_h]$ is given by
\begin{equation*}
  H'_h(\sigma_0)[\kappa_h] = \kappa_h|\nabla u_h(\sigma_0)|^2 + 2\sigma_0\nabla u_h(\sigma_0)\cdot\nabla u'_h(\sigma_0)[\kappa_h].
\end{equation*}

We need the following convergence result.
\begin{lemma}\label{lem:fem-W1q}
There exists some $q>2$ such that for any $\sigma\in \mathcal{A}$, $u_h(\sigma)\to u(\sigma)$ in $W^{1,q}(\Omega)$ as $h\to0^+$.
\end{lemma}
\begin{proof}
Let $u=u(\sigma)$ and $u_h=u_h(\sigma)$.
Then for any $v_h\in V_h$, by the triangle inequality, there holds
\begin{equation*}
  \|u-u_h\|_{W^{1,q}(\Omega)} \leq \|u-v_h\|_{W^{1,q}(\Omega)} + \|v_h-u_h\|_{W^{1,q}(\Omega)}.
\end{equation*}
It follows the inequality \eqref{eqn:inf-sup} and Galerkin orthogonality that with $q^{-1}+q^{\prime-1}=1$
\begin{align*}
  \|u_h-v_h\|_{W^{1,q}(\Omega)}& \leq C\sup_{0\neq \chi\in V_h}\frac{(\sigma\nabla (u_h-v_h),\nabla\chi) }{\|\chi\|_{W^{1,q'}(\Omega)}}\\
   & =  C\sup_{0\neq \chi\in V_h}\frac{(\sigma\nabla (u-v_h),\nabla\chi) }{\|\chi\|_{W^{1,q'}(\Omega)}} \leq C\|u-v_h\|_{W^{1,q}(\Omega)},
\end{align*}
where the last inequality is due to H\"{o}lder's inequality and the uniform bound on $\mathcal{A}$. Since
the choice of $v_h\in V$ is arbitrary, combining the last two estimates gives
\begin{equation*}
  \|u-u_h\|_{W^{1,q}(\Omega)} \leq C\inf_{v_h\in V_h}\|u-v_h\|_{W^{1,q}(\Omega)}.
\end{equation*}
Now the desired assertion follows by the density of $V_h$ in $W^{1,q}(\Omega)$.
\end{proof}

We have the following convergence for the finite element approximation.
\begin{lemma}\label{lem:polyhed-linear}
Let the sequence $\{\kappa_h\}_{h>0}\subset \mathcal{A}_{\sigma_0,h}$ converge
in $L^r(\Omega), \ r\geq 1$, to some $\kappa\in\mathcal{A}_\sigma$ as $h$ tends to zero.
Then $H_h'(\sigma_0)[\kappa_h] + H_h(\sigma_0)\to
H'(\sigma_0)[\kappa]+H(\sigma_0)$ in $L^1(\Omega)$ as $h\to0^+$.
\end{lemma}
\begin{proof}
By Lemma \ref{lem:fem-W1q},  $u_h\equiv u_h(\sigma_0) \to u(\sigma_0):=u$
in $W^{1,q}(\Omega)$ as $h\to 0^+$. Thus, $H_h(\sigma_0) = \sigma_0|\nabla u_h|^2 \to \sigma_0|\nabla u|^2
= H(\sigma_0)$ in $L^1(\Omega)$, and it suffices to show
$H_h'(\sigma_0)[\kappa_h]\to H'(\sigma_0)[\kappa]$. By the triangle inequality and $L^\infty(\Omega)$
bound on $\mathcal{A}_{\sigma_0}$ and $\mathcal{A}_{\sigma_0,h}$,
\begin{align*}
  \|H_h'(\sigma_0)[\kappa_h]-H'(\sigma_0)[\kappa]\|_{L^1(\Omega)}&\leq \|\kappa_h|\nabla u_h(\sigma_0)|^2-\kappa|\nabla u(\sigma_0)|^2\|_{L^1(\Omega)} \\
    & \quad + C\|\nabla u_h(\sigma_0)\cdot\nabla u_h'(\sigma_0)[\kappa_h]-\nabla u(\sigma_0)\cdot\nabla u'(\sigma_0)[\kappa]\|_{L^2(\Omega)}:={\rm I}+{\rm II}.
\end{align*}
By H\"{o}lder's inequality, the term ${\rm I}$ is bounded by (with $q>2$ from Lemma \ref{lem:dis-Meyers}, and $p^{-1}+2q^{-1}=1$)
\begin{align*}
  {\rm I} 
     & \leq \|\kappa_h-\kappa\|_{L^p(\Omega)}\|\nabla u_h(\sigma_0)\|_{L^q(\Omega)}^2 + \|\kappa\|_{L^\infty(\Omega)}\||\nabla u_h(\sigma_0)|^2-|\nabla u(\sigma_0)|^2\|_{L^1(\Omega)},
\end{align*}
where both  terms tend to zero, since $\kappa_h\to \kappa$ in $L^r(\Omega)$ and by Lemma \ref{lem:fem-W1q}, $ u_h(\sigma_0)\to u(\sigma_0)$ in $W^{1,q}(\Omega)$. Meanwhile, the term ${\rm
II}$ is bounded by
\begin{align*}
  {\rm II} 
  & \leq C\|\nabla (u_h(\sigma_0)-u(\sigma_0))\|_{L^2(\Omega)}\|\nabla u_h'(\sigma_0)[\kappa_h]\|_{L^2(\Omega)}\\
    &\quad +C\|\nabla u(\sigma_0)\|_{L^2(\Omega)}\|\nabla (u'_h(\sigma_0)[\kappa_h]- u'(\sigma_0)[\kappa])\|_{L^2(\Omega)}
  := {\rm III} + {\rm IV}.
\end{align*}
By the uniform bound on $\kappa_h$, $\|u_h'(\sigma_0)[\kappa_h]\|_{L^2(\Omega)}\leq C$ for some $C$ independent of $h$, and thus
 the term ${\rm III}\to 0$ as $h\to0^+$, in view of Lemma \ref{lem:fem-W1q}.
To bound the term ${\rm IV}$, let $w_h \in V_h$ satisfy
\begin{equation*}
   (\sigma_0\nabla w_h,\nabla \chi) = (\kappa\nabla u(\sigma_0),\nabla \chi)\quad \forall \chi\in V_h.
\end{equation*}
By Lemma \ref{lem:fem-W1q}, there holds $\|\nabla w_h - \nabla u'(\sigma_0)[\kappa]
\|_{L^2(\Omega)}\to 0$. Further, $v_h=w_h-u_h'(\sigma_0)[\kappa_h]\in V_h$ satisfies
\begin{align*}
  (\sigma_0\nabla v_h,\nabla \chi) = (\kappa_h\nabla u_h(\sigma_0)-\kappa\nabla u(\sigma_0),\nabla \chi)\quad \forall \chi\in V_h.
\end{align*}
Letting $\chi=v_h$ and applying Cauchy-Schwarz inequality lead to
\begin{equation*}
  \|\nabla v_h\|_{L^2(\Omega)}\le  C\|\kappa_h\nabla u_h(\sigma_0)-\kappa\nabla u(\sigma_0)\|_{L^2(\Omega)}.
\end{equation*}
Meanwhile, by Lemma \ref{lem:fem-W1q}, the
$L^\infty(\Omega)$ bound on $\kappa_h$ and the fact $\kappa_h\to \kappa$ in $L^1(\Omega)$,
\begin{equation}\label{eqn:res-to0}
  \|\kappa_h\nabla u_h(\sigma_0)-\kappa \nabla u(\sigma_0)\|_{L^2(\Omega)} \to 0.
\end{equation}
Combining these estimates shows ${\rm IV}\to0$ as $h\to 0^+$, which completes the proof of the lemma.
\end{proof}

Last, we state the convergence of the discrete approximations to the linearized functional $J_{\beta}$. The
proof is identical with that for Theorem \ref{thm:conv-fem}, but with Lemma \ref{lem:polyhed-linear} in place
of Lemma \ref{lem:polyhed}.
\begin{theorem}
The sequence $\{\kappa_h^*\in\mathcal{A}_{\sigma_0,h}\}_{h>0} $ of minimizers to the the discrete functionals $J_{\sigma_0,\beta,h}(\kappa_h)$ contains a subsequence converging in $L^1(\Omega)$ to a minimizer of the functional $J_{\sigma_0,\beta}(\kappa)$ as $h\to0^+$.
\end{theorem}

\bibliographystyle{abbrv}
\bibliography{main}

\end{document}